\documentclass[12pt]{amsart}

\newtheorem{theorem}{Theorem}[section]

\newtheorem{proposition}[theorem]{Proposition}

\newtheorem{corollary}[theorem]{Corollary}

\newtheorem{remark}[theorem]{Remark}

\newtheorem{lemma}[theorem]{Lemma}

\newtheorem{definition}[theorem]{Definition}

\newtheorem{conjecture}[theorem]{Conjecture}

\usepackage{amssymb}

\numberwithin{equation}{section}

\begin{document}

\baselineskip=16pt

\title[On coherent systems]{On coherent systems of type
$(n,d,n+1)$ on Petri curves}

\author{U. N. Bhosle, L. Brambila-Paz and P. E. Newstead}

\address{TIFR, Homi Bhabha Road, Mumbai 400005, India}
\email{usha@math.tifr.res.in}

\address{CIMAT, Apdo. Postal 402, C.P. 36240. Guanajuato, Gto,
M\'exico}
\email{lebp@cimat.mx}

\address{Department of Mathematical Sciences, The University of Liverpool,
Peach Street, Liverpool L69 7ZL, UK}
\email{newstead@liv.ac.uk}



\thanks{The authors are members of the research
group VBAC (Vector Bundles on Algebraic Curves). The first two authors were
supported by EPSRC grant GR/T22988/01 for a
visit to the University of Liverpool. The
second author acknowledges the support of CONACYT grant 48263-F.
The third author thanks CIMAT, Guanajuato, M\'exico and California State
University Channel Islands, where a part of this paper was
completed, and acknowledges support from the Academia Mexicana de Ciencias, under its exchange agreement with the Royal Society of London.}

\date{\today}

\begin{abstract}

We study coherent systems of type $(n,d,n+1)$ on a Petri curve $X$ of genus $g\ge2$. We describe the geometry of the moduli space of such coherent systems for large values of the parameter $\alpha$. We determine the top critical value of $\alpha$ and show that the corresponding ``flip'' has positive codimension. We investigate also the non-emptiness of the moduli space for smaller values of $\alpha$, proving in many cases that the condition for non-emptiness is the same as for large $\alpha$. We give some detailed results for $g\le5$ and applications to higher rank Brill-Noether theory and the stability of kernels of evaluation maps, thus proving Butler's conjecture in some cases in which it was not previously known.

\end{abstract}

\maketitle

\section{Introduction}\label{intro}

Let $X$ be a smooth irreducible projective
curve.
 A coherent system of type $(n,d,k)$ on $X$ is a pair $(E,V)$
where $E$ is a vector bundle on $X$ of rank $n$ and degree $d$
 and $V$ is a linear subspace of
$H^0(E)$ with $\dim V = k$. A notion of stability for coherent systems, dependent on a real variable $\alpha$, can be defined and leads to the construction of moduli spaces $G(\alpha;n,d,k)$ for $\alpha$-stable coherent systems (see \cite{an}, \cite{lep}, \cite{rag}). There is a natural compactification $\widetilde{G}(\alpha;n,d,k)$ obtained by considering equivalence classes of $\alpha$-semistable coherent systems. For $k=0$, $G(\alpha;n,d,0)$ is independent of $\alpha$ and coincides with the moduli space $M(n,d)$ of stable bundles of rank $n$ and degree $d$ on $X$, while $\widetilde{G}(\alpha;n,d,0)$ coincides with the corresponding moduli space $\widetilde{M}(n,d)$ of S-equivalence classes of semistable bundles. If $k\ge1$, a necessary condition for non-emptiness of $G(\alpha;n,d,k)$ (resp. $\widetilde{G}(\alpha;n,d,k)$) is $\alpha>0$ (resp. $\alpha\ge0$). For $n=1$, all coherent systems are $\alpha$-stable for all $\alpha>0$ and $G(\alpha;1,d,k)$ coincides with the classical variety of linear systems $G^{k-1}_d$.

A systematic study of coherent systems on curves of genus $g\ge2$ defined over the complex numbers was begun in \cite{bomn} (see also \cite{bo}) and continued in \cite{bgmmn2} and \cite{bgmmn}. In particular, precise conditions for non-emptiness of $G(\alpha;n,d,k)$ are known when $k\le n$ \cite[Theorem 3.3]{bgmmn2}. For $k>n$, much less is known. There are general results due to E.~Ballico \cite{ball} and M.~Teixidor i Bigas \cite{te2}; Teixidor's results are much the stronger, but are certainly not best possible. Some more detailed results have been obtained in \cite{le,la}. It is known that the $\alpha$-stability condition stabilises for $\alpha>d(n-1)$; we denote the corresponding ``large $\alpha$'' moduli space $G(\alpha;n,d,k)$ by $G_L(n,d,k)$ (see section \ref{prelim} for more details).

Our object in this paper is to study the case $k=n+1$ when the curve $X$ is a {\em Petri curve}, in other words, for every line bundle ${\mathcal L}$ on $X$, the multiplication map
$$H^0({\mathcal L})\otimes H^0({\mathcal L}^*\otimes K)\rightarrow H^0(K)$$
is injective. In this case $G_L:=G_L(\alpha;n,d,n+1)$ is non-empty if and only if the {\em Brill-Noether number}
$$\beta:=\beta(n,d,n+1)=g-(n+1)(n-d+g)$$
is non-negative \cite[Theorem 5.11]{bomn}. When in addition $d\le g+n$, $G(\alpha):=G(\alpha;n,d,n+1)$ is independent of $\alpha>0$ and its structure has been determined \cite[Theorem 2]{le}. Our first main theorem (Theorem \ref{t1}) generalises these results and gives a significant improvement of the estimate $\alpha>d(n-1)$ for $G(\alpha)$ to coincide with $G_L$. The detailed statement, which includes additional information on the structure of $G_L$, is as follows (here $E'$ denotes the subsheaf image of the evaluation map $V\otimes{\mathcal O}\rightarrow E$; for the definitions of {\em generated} and {\em generically generated}, see section \ref{prelim}).

\noindent{\bf Theorem \ref{t1}.} \begin{em} Suppose that $X$ is a Petri curve of genus $g\ge2$ and
$\alpha>\max\{0,\alpha_l\}$, where
\begin{equation*}
\alpha_l:=d(n-1)-n\left(n-1 +g -\left[\frac{g}{n}\right]\right).
\end{equation*}
Then

\begin{enumerate}

\item $G(\alpha) \not= \emptyset$ if and only if $\beta\geq 0$;

\item $G( \alpha ) =G_L$;

\item $(E,V) \in G(\alpha)$ if and only if $(E,V)$ is
generically generated
and $H^0(E'^*)=0$;

\item if $\beta>0$, $G(\alpha)$ is smooth and irreducible
of dimension $\beta$; moreover the generic element
of $G(\alpha)$ is generated;

\item if $\beta=0$, $G(\alpha)$ is a finite set of cardinality

$$
g!\prod_{i=0}^n\frac{i!}{(g-d+n+i)!};
$$
moreover every element of $G(\alpha)$ is generated.

\end{enumerate}\end{em}

It follows in particular that, if $(E,V)\in G_L$, then the cokernel $E/E'$ of the evaluation map $V\otimes{\mathcal O}\to E$ is a torsion sheaf. In section \ref{strat}, we define a stratification of $G_L$ in terms of the length of $E/E'$. More precisely, for every integer $t\ge0$, we write

$$\Sigma_t=\{(E,V)\in G_L:E/E'\mbox{ has length }t\}\ \ \mbox{and}\ \ S_t= \bigcup _{i\geq t } \Sigma _i.$$
Then 

\noindent{\bf Theorem \ref{t2}.}\begin{em}
Suppose $\beta\ge0$ and that the subsets $S_t$ of
$G_L$ are defined as above. Then
\begin{enumerate}
\item
$S_t$ is closed in $G_L$ and is non-empty if and only if
$0\le t\le t_1:=\left[ \frac{\beta}{n+1}\right]$;
\item for $1\le t\le t_1$,
$S_t\subset\overline{S_{t-1}\setminus S_t}$;
\item for $1\le t\le t_1$, $\dim S_t=\beta-t$;
\item $S_t$ is irreducible for $t<\frac{\beta}{n+1}$;
\item if $\frac{\beta}{n+1}$ is an integer,
then all irreducible components of $S_{t_1}$
have the same dimension.
\end{enumerate}\end{em}

In section \ref{top}, we show that there exists $(E,V)\in G_L$ such that $(E,V)$ is not $\alpha_l$-stable, in other words $\alpha_l$ is an (actual) critical value in the sense of \cite[Definition 2.4]{bomn}. In view of Theorem \ref{t1}, $\alpha_l$ is in fact the top critical value of $\alpha$.

Sections \ref{anyalpha} -- \ref{low} are concerned with the moduli space $G(\alpha)$ for arbitrary $\alpha$. It was proved in \cite{le} that, if $G(\alpha)\ne\emptyset$, then $\beta\ge0$. Several results on the non-emptiness of $G(\alpha)$ when $\beta\ge0$ were also proved in \cite{le}. In section \ref{anyalpha}, we extend these results using the techniques of elementary transformations and extensions of coherent systems. In particular for $n=2,3,4$, we show in section \ref{n=23} that $G(\alpha)\ne\emptyset$ if and only if $\beta\ge0$ (see Theorems \ref{th2}, \ref{th3} and \ref{th4} for details). We then consider in section \ref{low} the case $g\le5$ (including $g=0$ and $g=1$, which have been excluded from our general discussion). For $g\le2$, the results are complete, while for $g=3,4,5$, there are a few cases still to be solved.

In section \ref{apli}, we give some applications to higher rank Brill-Noether theory (see section \ref{prelim} for definitions). We first obtain some irreducibility and smoothness results for Brill-Noether loci using the programme envisaged in \cite[section 11]{bomn}. For the second application, suppose that ${\mathcal L}$ is a generated line bundle of degree $d>0$ and let $V$ be a linear subspace of $H^0({\mathcal L})$ of dimension $n+1$ which generates ${\mathcal L}$ (in other words, $({\mathcal L},V)$ is a generated coherent system of type $(1,d,n+1)$). We have an evaluation sequence
$$0\longrightarrow M_{V,{\mathcal L}}\longrightarrow V\otimes{\mathcal O}\longrightarrow {\mathcal L}\longrightarrow0.$$
The bundles $M_{V,{\mathcal L}}$ arise in several contexts and have been used in the study of Picard bundles \cite{el}, normal generation of vector bundles \cite{pr,bu1}, syzygies and projective embeddings \cite{green1}, higher rank Brill-Noether loci \cite{mer1}, theta-divisors \cite{bea,mis} and coherent systems \cite{bu,bomn,le}. 

A particular point of interest is to determine whether or not $M_{V,{\mathcal L}}$ is stable. In fact, in \cite{bu}, Butler 
conjectured that $M_{V,{\mathcal L}}$ is stable for general choices of $X$, $\mathcal L$ and $V$. His conjecture \cite[Conjecture 2]{bu} is concerned more generally with generated coherent systems of any type $(n,d,k)$. We shall be concerned only with the case $n=1$; Butler's conjecture can then be stated as follows.

\noindent{\bf Conjecture \ref{conj}.} \begin{em} Let $X$ be a Petri curve of genus $g\ge3$. Suppose that 
$\beta:=\beta(1,d,n+1)\ge0$ and that $\mathcal L$ is a general element of $B(1,d,n+1)$ (when $\beta=0$, $\mathcal L$ 
can be any element of the finite set $B(1,d,n+1)$) and let $V$ be a general subspace of $H^0(\mathcal L)$ of dimension $n+1$. Then $M_{V,{\mathcal L}}$ is stable.\end{em}

In most of the above references, $V$ is taken to be $H^0({\mathcal L})$, which implies by Riemann-Roch that $d\le g+n$ and the stability problem has been solved in this case \cite{bu,le}. However the case where $V$ is a proper subspace of $H^0({\mathcal L})$ seems equally interesting; this is mentioned but not used in \cite{bu}, used in a minor way in \cite{bomn} and studied for low values of the codimension in \cite{mis}. However, the restriction placed on $d$ in \cite{mis} implies that $d\le2n$, so this case (although not the remaining results of \cite{mis}) is also covered in \cite{mer1,mer2}. In the present paper, we do not use the stability of $M_{V,{\mathcal L}}$ except through citations from earlier papers. We are therefore able to use our methods to prove the stability of $M_{V,{\mathcal L}}$ in some cases where it is not (to our knowledge) already known.  These new examples for which $M_{V,{\mathcal L}}$ is stable depend essentially on the use of extensions of coherent systems (more specifically on Propositions \ref{prop5}, \ref{prop6}, \ref{prop7}, \ref{prop32} and \ref{prop33}).

We assume throughout that $X$ is a Petri curve of genus $g$, where, except in section \ref{low}, $g\ge2$. We assume also that $X$ is defined over the complex numbers. We denote the canonical line bundle on $X$ by $K$. 

\section{Preliminaries}\label{prelim}

In this section, we recall some facts about coherent systems, most of which can be found in \cite{bomn} and \cite{he}.

For  $\alpha \in \mathbb{R}$, we define the
{\em $\alpha$-slope} of the coherent system $(E,V)$ of type $(n,d,k)$ by
$$\mu _{\alpha}(E,V):= \frac{d}{n} + \alpha \frac{k}{n}.$$
A {\em coherent subsystem} of $(E,V)$ is a pair $(F,W)$, where $F$ is a subbundle of $E$ and $W\subset V\cap H^0(F)$.

\begin{definition}\label{def1}\begin{em}
For any $\alpha \in \mathbb{R}$, a coherent system $(E,V)$ on $X$ is
$\alpha$-{\em stable}
(respectively $\alpha$-{\em semistable}) if, for every proper coherent subsystem $(F,W)$,
$$\mu _{\alpha}(F,W)<\mu _{\alpha}(E,V) \ \ \ ({\rm respectively} \leq ).$$
\end{em}\end{definition}
We denote by $G(\alpha ;n,d,k)$ the moduli space of $\alpha$-stable
coherent systems of type $(n,d,k)$ (\cite{an}, \cite{lep}, \cite{rag}) and by $\widetilde{G}(\alpha;n,d,k)$ the moduli space of S-equivalence classes of $\alpha$-semistable coherent systems (see \cite[section 2]{bomn}).
It follows from the definition of $\alpha $-stability that,
if $k\ge1$ and $G(\alpha;n,d,k)\not=\emptyset$, then $\alpha > 0$ and $d>0$ \cite[section 2 and Lemmas 4.1 and 4.3]{bomn}.

\begin{remark}\begin{em}\label{R8}
Given a coherent system $(E,V)$ and an effective line bundle ${\mathcal L}$,
let ${\widetilde E} = E\otimes {\mathcal L}$. Choose a non-zero section $s$ of
${\mathcal L}$ and let ${\widetilde V}$ be the image of $V$ in $H^0(\widetilde E)$
under the induced
inclusion $H^0(E) \hookrightarrow H^0(\widetilde E):v\mapsto v\otimes s$. Then
\begin{itemize}
\item[(1)]
 $E$ is (semi)stable if and only if ${\widetilde E}$ is (semi)stable.
\item[(2)] $(E,V)$ is $\alpha $-(semi)stable if and only if $({\widetilde E},
{\widetilde V})$ is $\alpha $-(semi)stable \cite[Lemma 1.5]{rag}.
\end{itemize}\end{em}\end{remark}

\begin{remark}\begin{em}\label{bint}
It follows from Remark \ref{R8} that, if
 $G(\alpha ;n,d,k)\not= \emptyset$ for all integers $d\in [a,b]$ with
 $a,b\in{\mathbb Z}$
and $b-a \geq n-1$, then $G(\alpha ;n,d,k)\not= \emptyset$ for all
$d\geq a.$
\end{em}\end{remark}

For any triple $(n,d,k)$, we define the {\em Brill-Noether number} $\beta(n,d,k)$ by
$$\beta(n,d,k)=n^2(g-1) +1-k(k-d+n(g-1)).$$
For a coherent system $(E,V)$, the {\em Petri map} at $(E,V)$ is the map
\begin{equation}\label{pet1}
V\otimes H^0(E^*\otimes K)\rightarrow
H^0(E\otimes E^*\otimes K)
\end{equation}
given by multiplication of sections. We have the following fundamental result (see \cite[Corollaire 3.14]{he}, \cite[Corollary 3.6 and Proposition 3.10]{bomn}).

\begin{proposition}\label{e}
Every irreducible component of $G(\alpha;n,d,k)$ has dimension $\ge\beta(n,d,k)$. Moreover, if
$(E,V)\in G(\alpha ;n,d,k)$, then $G(\alpha ;n,d,k)$ is smooth of dimension $\beta(n,d,k)$ at $(E,V)$ if and only if (\ref{pet1}) is
injective.
\end{proposition}

For a line bundle ${\mathcal L}$ with $V=H^0({\mathcal L})$, the Petri map (\ref{pet1}) takes the form
\begin{equation}\label{pet2}
H^0({\mathcal L})\otimes H^0({\mathcal L}^*\otimes K)\to H^0(K)
\end{equation}

\begin{definition}\label{def2}
\begin{em}The curve $X$ is a {\em Petri curve} if (\ref{pet2}) is injective for every line bundle $\mathcal L$ on $X$.
\end{em}\end{definition}

It is a classical fact (see \cite{arb}) that the general curve of any given genus $g$ is a Petri curve. It should however be emphasised that, except for certain low values of the genus, there exist $\alpha$-stable coherent systems $(E,V)$ on the general curve for which (\ref{pet1}) is not injective (see, for example, \cite[\S5]{te1}).

The $\alpha$-range is divided into a finite set
of intervals by a set of {\em critical values $\{\alpha_i\}$}, where, for $k\ge n$, $$
0=\alpha_0 < \alpha_1< \cdots <\alpha_L < \infty
$$
\cite[Proposition 4.6]{bomn}. For $ \alpha , \alpha' \in ( \alpha_i,
\alpha_{i+1})$, we have
$G( \alpha; n,d,k) = G( \alpha'; n,d,k)$ and we denote this moduli space by
$G_i: = G_i(n,d,k)$. In particular, for
$\alpha>\alpha_L$, we have the ``large $\alpha$'' moduli
space $G_L:=G_L(n,d,k)$.

The relation between two consecutive moduli spaces
$G_{i-1}$ and $G_i$ is given by the so called ``flips'' (see \cite{bomn} for a more complete description).
For any critical value $\alpha_i$, we denote
by $\alpha_i^-$, $\alpha_i^+$ values of $\alpha$
in the intervals respectively immediately before and
after $\alpha_i$ and let
$$
G_i^+:= \{ (E,V)\in G_i\,|\,(E,V)\ {\rm is\ not }
\  \alpha_i^-{\rm-stable}\}
$$
and
$$
G_i^- = \{ (E,V) \in G_{i-1}\,|\,(E,V)\
{\rm is\ not } \  \alpha_i^+{\rm-stable}\}.
$$
These are called {\em flip loci} and
\begin{equation}\label{eqn:iso}
G_i - G_i^+ = G_{i-1} - G_i^-.
\end{equation}

For any critical value $\alpha_i$, the flip
locus $G_i^+$ consists of the coherent systems
$(E,V)\in G_i$ for which there exists an exact sequence
\begin{equation}\label{eqif}
0 \to (E_1,V_1) \to (E,V) \to (E_2,V_2) \to 0,
\end{equation}
with $(E_j,V_j)$ of type $(n_j,d_j,k_j)$, $\alpha_i$-semistable
 and  $\alpha_i^+$-stable for $j=1,2$ and
\begin{equation}\label{eqn:alpha}
\mu_{\alpha_i} (E_1,V_1)
= \mu_{\alpha_i}(E_2,V_2), \ \ k_1/n_1 < k/n
\end{equation}
(see \cite[Lemma 6.5]{bomn}  for more details). Similarly, the flip
locus $G_i^-$ consists of the coherent systems
$(E,V)\in G_{i-1}$ for which there exists an exact sequence
\begin{equation*}
0 \to (E_2,V_2) \to (E,V) \to (E_1,V_1) \to 0,
\end{equation*}
with $(E_j,V_j)$ $\alpha_i$-semistable
 and  $\alpha_i^-$-stable for $j=1,2$ and satisfying (\ref{eqn:alpha}).

In \cite{bomn}, numerical criteria were obtained to help determine whether the flip loci have positive codimension. More generally, these criteria can be used to estimate the number of parameters on which the coherent systems $(E,V)$ given by extensions (\ref{eqif}) depend. Define, for $\{j,l\}=\{1,2\}$,
\begin{eqnarray}\label{c12}
C_{jl}&=&n_jn_l(g-1)-n_jd_l+n_ld_j+k_jd_l-k_jn_l(g-1)-k_jk_l\nonumber\\
&=&(k_j-n_j)(d_l-n_l(g-1))+n_ld_j-k_jk_l
\end{eqnarray}
and
\begin{equation}\label{h2}
{\mathbb H}^0_{jl}={\rm Hom}((E_j,V_j),(E_l,V_l)),\ \ {\mathbb H}^2_{jl}=H^0(E_l^*\otimes N_j\otimes K)^*,
\end{equation}
$N_j$ being the kernel of the evaluation map
$V_j\otimes\mathcal{O}\to E_j$. We have, by \cite[equations (8) and (11)]{bomn},
\begin{equation}\label{eqn:200}
\dim\mbox{Ext}^1((E_j,V_j),(E_l,V_l))=C_{jl}+\dim{\mathbb H}^0_{jl}+\dim{\mathbb H}^2_{jl}.
\end{equation}

The following lemma can be regarded as a simplified version of
\cite[Lemma 6.8]{bomn}.

\begin{lemma}\label{alpha+}
Suppose that, for $j=1,2$, $(E_j,V_j)$ has type $(n_j,d_j,k_j)$ and varies in a family depending on at most 
$\beta(n_j,d_j,k_j)$ parameters. Suppose further that, for some $h_0$, $h_2$,
$$\dim{\mathbb H}^0_{21}\le h_0,\ \ \dim{\mathbb H}^2_{21}\le h_2$$
for all $(E_j,V_j)$ occurring in these families and that
$$
C_{12} -h_0-h_2>0.
$$
Then the coherent systems $(E,V)$ arising as non-trivial extensions of the form (\ref{eqif}) depend on at most
$\beta(n,d,k)-1$ parameters.
\end{lemma}

\begin{proof} By (\ref{eqn:200}), for fixed $(E_1,V_1)$, $(E_2,V_2)$, the coherent systems $(E,V)$ depend on at most
$$C_{21}+h_0+h_2-1$$ parameters. The result follows from \cite[Corollary 3.7]{bomn}.
\end{proof} 

\begin{remark}\begin{em}\label{rmk100}
Note that, if we assume in addition that $(E,V)$ is $\alpha$-stable for some $\alpha$, then we can take $h_0=0$, since a non-zero homomorphism $(E_2,V_2)\to(E_1,V_1)$ would contradict \cite[Proposition 2.2(ii)]{bomn}.
\end{em}\end{remark}

The ``small $\alpha$'' moduli spaces $G_0(n,d,k)$ and $\widetilde{G}_0(n,d,k)$ are closely related to the  {\em Brill-Noether locus} $B(n,d,k)$ of stable bundles, which
is defined by
$$B(n,d,k):=\{ E\in M(n,d)| h^0(E)\geq k\}.$$
Similarly one defines the Brill-Noether locus $\widetilde{B}(n,d,k)$
for semistable bundles by
$$\widetilde{B}(n,d,k):=\{ [E]\in \widetilde{M}(n,d)| h^0(\mbox{gr}(E))\geq k\},$$
where $\widetilde{M}(n,d)$ is the moduli space of S-equivalence classes of semistable bundles, $[E]$ is the S-equivalence class of $E$ and $\mbox{gr}(E)$ is the graded object associated to a semistable bundle $E$. The formula $(E,V)\mapsto[E]$ defines a morphism
$$\psi : G_0(n,d,k)\rightarrow \widetilde{B}(n,d,k),$$
whose image contains $B(n,d,k)$. We shall use this morphism $\psi$ in section \ref{apli}.

We finish this section with a useful definition and some notation.

\begin{definition}\label{def3}\begin{em}
A coherent system $(E,V)$ is 

{\em generated} if the evaluation map $V\otimes {\mathcal O}\to E$ is surjective;

{\em generically generated} if the cokernel of the evaluation map is a torsion sheaf.
\end{em}\end{definition}

\noindent{\bf Notation.} We shall write $\beta$, $G(\alpha)$, $\widetilde{G}(\alpha)$, $G_L$ for $\beta(n,d,n+1)$, $G(\alpha;n,d,n+1)$, $\widetilde{G}(\alpha;n,d,n+1)$, $G_L(n,d,n+1)$ respectively. For any coherent system $(E,V)$, we shall consistently denote by $E'$ the subsheaf image of the evaluation map. We shall also denote by $(n_i,d_i,k_i)$ the type of a coherent system $(E_i,V_i)$.

\section{The moduli space for large $\alpha$}\label{large}

In this section we assume that $X$ is a Petri curve and obtain a
strengthening of \cite[Theorem 5.11]{bomn}.
 In particular we obtain a much better lower bound
on the parameter $\alpha$ which ensures that $G(\alpha)=G_L$. In
later sections we shall prove that this bound is best possible and
describe a natural stratification of $G_L$. For  $d\le g+n$,
Theorem \ref{t1} has been proved in \cite[Theorem 2]{le}. We recall that, for any coherent system $(E,V)$,  $E'$ denotes
the subsheaf image of $V\otimes\mathcal{O}$ in $E$.

\begin{theorem}\label{t1} Suppose that $X$ is a Petri curve and
$\alpha>\max\{0,\alpha_l\}$, where
\begin{equation}\label{alpha0}
\alpha_l:=d(n-1)-n\left(n-1 +g -\left[\frac{g}{n}\right]\right).
\end{equation}
Then

\begin{enumerate}

\item $G(\alpha) \not= \emptyset$ if and only if $\beta\geq 0$;

\item $G( \alpha ) =G_L$;

\item $(E,V) \in G(\alpha)$ if and only if $(E,V)$ is
generically generated
and $H^0(E'^*)=0$;

\item if $\beta>0$, $G(\alpha)$ is smooth and irreducible
of dimension $\beta$; moreover the generic element
of $G(\alpha)$ is generated;

\item if $\beta=0$, $G(\alpha)$ is a finite set of cardinality

$$
g!\prod_{i=0}^n\frac{i!}{(g-d+n+i)!};
$$
moreover every element of $G(\alpha)$ is generated.

\end{enumerate}

\end{theorem}

We shall prove Theorem \ref{t1} by means of a sequence
of propositions.  We begin with two lemmas, the first of which is a
variant of \cite[Lemma 3.1]{le}. Since the hypotheses are not exactly
the same as those of \cite[Lemma 3.1]{le}, we include a proof.

\begin{lemma}\label{l1} Let $X$ be a Petri curve
and $(E,V)$ a coherent system of type $(n,d,k)$.
If $(E,V)$ is generically generated
and $H^0(E'^*)=0$, then $k\ge n+1$ and $d\geq
g+n-\left[\frac{g}{n+1}\right]$. Moreover,
if $(E_2,V_2)$ is a quotient coherent system of
 $(E,V)$, then $(E_2,V_2)$
is generically generated and $H^0(E_2'^*)=0$.

\end{lemma}

\begin{proof} Certainly $k\ge n$. If $k=n$, then
$E'\cong\mathcal{O}^n$, contradicting the
hypothesis $H^0(E'^*)=0$. So $k\ge n+1$.

Replacing $V$, if necessary, by a subspace of
dimension $n+1$ which generates $E'$, we have an exact sequence

\begin{equation}\label{eqi2}
0\rightarrow {\mathcal L}^*\rightarrow V\otimes \mathcal{O}
\rightarrow E'\rightarrow 0,
\end{equation}
where ${\mathcal L}=\det E'$. From the dual of (\ref{eqi2})
and the hypothesis $H^0(E'^*)=0$, we see that
$h^0({\mathcal L})\ge n+1$. By classical Brill-Noether theory, this implies that

$$
\deg E'=\deg {\mathcal L}\ge \frac{ng}{n+1}+n=g+n-\frac{g}{n+1}.
$$
Hence $d\ge\deg E'\ge g+n-\left[\frac{g}{n+1}\right]$ as required.

For the last part, note that the image of $E'$ in
$E_2$ is precisely $E_2'$. Hence $E_2'$ is a
quotient of $E'$ and the result follows.
\end{proof}

\begin{remark}\label{r1}\begin{em} Note that
\begin{equation}\label{alphal}
\alpha_l=(n-1)(d-g-n)-\left(g-n\left[\frac{g}{n}\right]\right)=(n-1)(d-n)
-n\left(g-\left[\frac{g}{n}\right]\right).
\end{equation}
and that
$$
d\geq g+n-\left[\frac{g}{n+1}\right]\Leftrightarrow
d\ge\frac{ng}{n+1}+n \Leftrightarrow\beta\ge0.
$$
Note in particular that, by (\ref{alphal}),
$$\alpha_l\ge0\Rightarrow d\ge g+n\Rightarrow \beta\ge0.$$

\end{em}\end{remark}

\begin{lemma}\label{l2} Let $f:{\mathbb Z}_{>0}
\to{\mathbb Q}$ be defined by
$$
f(r):=\frac1r\left(g-\left[\frac{g}{r+1}\right]\right).
$$
Then $f$ is a decreasing function of $r$.
\end{lemma}

\begin{proof} If $g\ge r+1$, we have
$$
f(r)\ge\frac1r\left(g-\frac{g}{r+1}\right)=\frac{g}{r+1}
$$
and
$$
f(r+1)\le\frac1{r+1}\left(g-\frac{g-r-1}{r+2}\right)
=\frac{g+1}{r+2}\le\frac{g}{r+1}.
$$
On the other hand, if $g<r+1$, then
$$
f(r)=\frac{g}{r}>\frac{g}{r+1}=f(r+1).
$$
\end{proof}

\begin{proposition}\label{pe1} Suppose that $(E,V)$ is
 a generically generated coherent system of type
$(n,d,n+1)$ and $H^0(E'^*)=0$.  Then  $(E,V)$ is $\alpha $-stable for
$\alpha >\max\{0,\alpha_l\}$.

\end{proposition}

\begin{proof} Let $(E_2,V_2)$ be a proper quotient
coherent system of $(E,V)$ of type $(n_2,d_2,k_2)$.
It follows from Lemma \ref{l1} that  $k_2\ge n_2+1$
and $d_2\geq g+n_2-\left[\frac{g}{n_2+1}\right]$. Hence

\begin{equation}\label{eq1}
\mu_{\alpha}(E_2,V_2)\geq1+\frac1{n_2}
\left(g-\left[ \frac{g}{n_2+1}\right]\right)
 +\alpha \left(\frac{n_2+1}{n_2}\right).
\end{equation}

If $\alpha >\max\{0,\alpha_l\}$ then, since $0<n_2<n$,

\begin{equation}\label{ineq}
\alpha \left(\frac{1}{n_2}-\frac{1}{n}\right)=
\alpha \left(\frac{n-n_2}{nn_2}\right)
\geq \frac{\alpha}{n(n-1)}> \frac{d}{n}
 -1-\frac1{n-1}\left(g-\left[\frac{g}{n}\right]\right).
\end{equation}
Hence, from (\ref{eq1}) and Lemma \ref{l2},
$$
\mu_{\alpha}(E_2,V_2)-\mu_{\alpha}(E,V)>
\frac1{n_2}\left(g-\left[\frac{g}{n_2+1}
\right]\right)-\frac1{n-1}\left(g-\left[\frac{g}{n}
\right]\right)\ge0.
$$
Since this holds for all $(E_2,V_2)$, it follows that
$(E,V)$ is $\alpha$-stable.
\end{proof}

\begin{remark}\label{rmk:200}\begin{em}
Suppose $(E_2,V_2)$ is a coherent system of type $(n_2,d_2,k_2)$ with
$$0<n_2<n,\ \ k_2\ge n_2+1,\ \ d_2\ge g+n_2-\left[\frac{g}{n_2+1}\right].$$
If $\alpha\ge\alpha_l>0$, then (\ref{eq1}) still holds as does the first inequality in (\ref{ineq}), while the second inequality in (\ref{ineq}) becomes $\ge$. So
$$\mu_\alpha(E_2,V_2)\ge\mu_\alpha(E,V)$$
with equality if and only if $\alpha=\alpha_l$ and

$$n_2=n-1,\ \ k_2=n,\ \ d_2=g+n-1-\left[\frac{g}n\right].$$
\end{em}\end{remark}

\begin{proposition}\label{pe2} For given $n$ and $d$,
 the following three conditions are equivalent:

\begin{itemize}

\item[{\rm(a)}] there exists a generated coherent
system $(E,V)$ of type $(n,d,n+1)$ with $H^0(E^*)=0$;

\item[{\rm(b)}] there exists a generically generated
 coherent system $(E,V)$ of type $(n,d,n+1)$ with $H^0(E'^*)=0$;

\item[{\rm(c)}] $\beta  \geq 0$.

\end{itemize}

\end{proposition}

\begin{proof} Clearly (a) implies (b) and, by Lemma
\ref{l1} and Remark \ref{r1}, (b) implies (c).

Now suppose (c) holds. By classical Brill-Noether
theory, $G(1,d,n+1)\ne\emptyset$ and its general
element $({\mathcal L},W)$ is generated (in the case $\beta=0$, $G(1,d,n+1)$ is finite and all elements 
are generated). If we define $E$ by the exact sequence
$$
0\rightarrow E^*\rightarrow W\otimes \mathcal{O}
\rightarrow {\mathcal L}\rightarrow 0,
$$
then $(E,W^*)$ satisfies (a).
\end{proof}

\begin{proposition}\label{pe3} Suppose that
$\alpha >\max\{0,\alpha_l\}$ and
$(E,V)$ is an $\alpha$-semistable coherent
system of type $(n,d,n+1)$.
Then $(E,V)$ is generically generated and
$H^0(E'^*)=0$.

\end{proposition}

\begin{proof}Since $(E',V)$ is a generated
coherent system, we can write
$(E',V)\cong (\mathcal{O}^s,H^0(\mathcal{O}^s))\oplus (G,W)$ where $H^0(G^*)=0$, $W=H^0(G)\cap V$ and $(G,W)$ is generated.
Let $r$ denote the rank of $G$.  Note that,
since $h^0(E')\geq n+1$, we must have $r\ge1$.
 We require to show that $r=n$.

Suppose to the contrary that $r\le n-1$. Since the coherent
system $(G,W )$ is generated, we have,
by Lemma \ref{l1}, $\deg G\geq g+r-\left[\frac{g}{r+1}\right]$.
Hence
$$
\frac1r\left(g-\left[\frac{g}{r+1}\right]\right) +1
+\alpha \frac{n+1-s}{r}\leq \mu_{\alpha}(G,W).
$$
Since $(E,V)$ is $\alpha$-semistable, it follows that
$$
\frac1r\left(g-\left[\frac{g}{r+1}\right]\right)  +1
 +\alpha \frac{n+1-s}{r}\leq \frac{d}{n} +\alpha\frac{n+1}{n}.
$$
Now $s\le n-r$; so, for any fixed $r$, the minimum value
for the left-hand side of this inequality is given by
$s=n-r$. By Lemma \ref{l2}, this minimum value is then a
decreasing function of $r$. Hence
$$
\frac1{n-1}\left(g-\left[\frac{g}{n}\right]\right)
+1 +\alpha \frac{n}{n-1}\leq \frac{d}{n} +\alpha\frac{n+1}{n},
$$
i.~e.
$$
\frac{\alpha}{n(n-1)}\leq \frac{d-n}{n}-\frac1{n-1}
\left(g-\left[\frac{g}{n}\right]\right),
$$
contradicting the hypothesis that $\alpha >\alpha_l$.
\end{proof}

\begin{remark}\begin{em}\label{re2} Under the hypotheses
 of Proposition \ref{pe3}, we have an exact sequence
\begin{equation}\label{tau}
0\rightarrow E'\rightarrow E\rightarrow \tau\rightarrow 0,
\end{equation}
where $\tau$ is a torsion sheaf.
If $t$ is the length of $\tau$, then $\deg E'=d-t$. Since
$(E',V)$ is generated and $H^0(E'^*)=0$, Lemma \ref{l1} gives
$d-t \geq g+n-\left[\frac{g}{n+1}\right]$, or equivalently
\begin{equation}\label{eqt}
t\le t_1:=d-g-n+\left[\frac{g}{n+1}\right]=
\left[\frac{\beta}{n+1}\right].
\end{equation}
We shall see later (Theorem \ref{t2}) that this
bound is best possible. In particular, if we write
$$
d_0=g+n-\left[\frac{g}{n+1}\right],
$$
then, for $d>d_0$, we have $t_1\ge1$, so there exists a
non-generated coherent system $(E,V)$ in $G_L$.
\end{em}\end{remark}

\begin{proof}[Proof of Theorem \ref{t1}]
Parts (2) and (3) follow from Propositions \ref{pe1} and
\ref{pe3}, and (1) then follows from Proposition \ref{pe2}.

(4) If $\beta>0$, it follows from
\cite[Lemma 4.2]{le}and \cite[Theorem 5.11]{bomn} that
$G(\alpha)$ is smooth and
irreducible of dimension $\beta$. The fact
that the generic element is
generated then follows from
Proposition \ref{pe2}.

(5) If $\beta=0$, it follows from
\cite[Lemma 4.2]{le} that $G(\alpha)$
is finite and that, as a scheme, it is reduced. By (\ref{tau})
and (\ref{eqt}), every element is generated.
The formula for the
cardinality of $G(\alpha)$ now follows from
\cite[Chapter V, formula (1.2)]{arb}.
\end{proof}

\section{A stratification of $G_L$}\label{strat}

Let
\begin{equation}\label{sigma0}
\Sigma_0=\{ (E,V)\in G_L | (E,V)\ {\rm is\ generated}\}.
\end{equation}
Clearly $\Sigma_0$ is open in $G_L$. If $\beta\ge0$, we know
from Theorem \ref{t1} that
$\Sigma_0\not= \emptyset$.  Moreover, by Remark \ref{re2}, the
complement of $\Sigma_0$ in $G_L$ is a disjoint
union of locally closed subsets $\Sigma_t$,
defined for $1\le t\le t_1$ by
\begin{equation}\label{sigmat}
\Sigma_t=\{ (E,V)\in G_L |\,\exists\
{\rm an\ exact\ sequence}\ (\ref{tau})\
{\rm with}\ \tau\ {\rm of\ length}\ t\}.
\end{equation}

We now define
$$
S_t= \bigcup _{i\geq t } \Sigma _i,
$$
where the $\Sigma_i$ are the locally
closed subsets of $G_L$ defined in
(\ref{sigma0}) and (\ref{sigmat}).
Clearly $G_L=S_0 \supset S_1 \supset
\dots \supset S_{t}\supset\dots$.
We would like to show that the subsets
$S_t$ define a well-behaved stratification of $G_L$.

We begin with a lemma, which will be needed again later
\begin{lemma}\label{lem:200}
Suppose that we have an exact sequence
$$0\longrightarrow F\longrightarrow E\longrightarrow\tau\longrightarrow0,$$
where $\tau$ is a torsion sheaf of length $t$, and that $V$ is a subspace of $H^0(F)$ of dimension $n+1$. Then
$$(E,V)\in G_L(n,d,n+1)\Leftrightarrow(F,V)\in G_L(n,d-t,n+1).$$
\end{lemma}

\begin{proof} It is clear that $(E,V)$ is generically generated if and only if $(F,V)$ is generically generated and that $E'=F'$. The result follows at once from Theorem \ref{t1}(3).
\end{proof}

\begin{theorem}\label{t2}
Suppose $\beta\ge0$ and that the subsets $S_t$ of
$G_L$ are defined as above. Then
\begin{enumerate}
\item
$S_t$ is closed in $G_L$ and is non-empty if and only if
$0\le t\le t_1:=\left[ \frac{\beta}{n+1}\right]$;
\item for $1\le t\le t_1$,
$S_t\subset\overline{S_{t-1}\setminus S_t}$;
\item for $1\le t\le t_1$, $\dim S_t=\beta-t$;
\item $S_t$ is irreducible for $t<\frac{\beta}{n+1}$;
\item if $\frac{\beta}{n+1}$ is an integer,
then all irreducible components of $S_{t_1}$
have the same dimension.
\end{enumerate}\end{theorem}

\begin{proof} The fact that $S_t$ is empty if
$t> t_1=\left[ \frac{\beta}{n+1}\right]$ has already been
proved in Remark \ref{re2}. We prove the rest of the theorem by
induction on $t_1$, the result being an immediate consequence
of Theorem \ref{t1} if $t_1=0$.

Suppose therefore that $t_1\ge 1$. We consider the moduli space
$$
G_{L,d-1}:=G_L(n,d-1,n+1)
$$
and denote by $S_{t,d-1}$ the subset of $G_{L,d-1}$ given by
$$
S_{t,d-1}:= \{ (F,V)\in G_{L,d-1} |\,\exists\
 {\rm an\ exact\ sequence}\ (\ref{tau})\ {\rm with}\
\tau\ {\rm of\ length}\ \ge t\}.
$$
The maximum value of $t$ on $G_{L,d-1}$ is
$$
\left[\frac{\beta(1,d-1,n+1)}{n+1}\right]=t_1-1,
$$
so we can assume inductively that the theorem holds for $G_{L,d-1}$.

Note next that, if $(F,V)\in G_{L,d-1}$ and $E$ is defined by an
elementary transformation
\begin{equation}\label{eqn2}
0\rightarrow F\rightarrow E\rightarrow \tau\rightarrow0,
\end{equation}
with $\tau$ a torsion sheaf of length $1$, then $(E,V)\in G_L$ by
Lemma \ref{lem:200}. In fact it is easy to see that the $(E,V)$
obtained in this way are precisely the elements of $S_1$ and, more
generally, for $1\le t\le t_1$,
\begin{equation}\label{elem}
(E,V)\in S_t\Leftrightarrow (F,V)\in S_{t-1,d-1}.
\end{equation}

The next step is to carry out this construction for families
of coherent systems.  Since $(n,d-1,n+1)$ are coprime
there is a universal family
$(\mathcal{U},\mathcal {V})$
parametrised by $G_{L,d-1}$ \cite[Proposition A.8]{bgmmn2}.
Denote by $p:\mathbb{P}\mathcal{U}\rightarrow X\times
 G_{L,d-1}$ the natural projection.
As in the Hecke correspondence of \cite{nr},
$\mathbb{P}\mathcal{U}$ parametrises the triples
$$
(F,V,0\rightarrow F\rightarrow E\rightarrow\tau\rightarrow0)
$$
for which $(F,V)\in G_{L,d-1}$ and
$\tau$ has length $1$. The universal property of
$G_L$ now gives us a diagram
$$
\begin{array}{ccl}\mathbb{P}\mathcal{U}&
\stackrel{\Psi}{\longrightarrow}
&G_L\\p\downarrow\ \ &&\\\ \ \ \ \ \ \ X\times G_{L,d-1}.&&
\end{array}
$$
By (\ref{elem}), we have
\begin{equation}\label{st}
S_t= \Psi (p^{-1}(X\times S_{t-1,d-1})),\ \
 \Psi^{-1}(S_{t-1}\setminus
S_t)=p^{-1}(X\times (S_{t-2,d-1}\setminus S_{t-1,d-1})).
\end{equation}
The fact that $S_t\ne\emptyset$ for $t\le t_1$ follows at once.
Moreover $G_{L,d-1}$ is a projective variety and, by inductive
hypothesis, $S_{t-1,d-1}$ is  closed and,
 provided $t-1<\frac{\beta}{n+1}-1$, also irreducible;
hence $S_t$ is closed in $G_L$, completing the proof of (1).
Properties (2) and (4) follow immediately from (\ref{st}).

For (3), note that, by the inductive hypothesis,
\begin{equation}\label{dim}
\dim(p^{-1}(X\times S_{t-1,d-1}))=\beta(n,d-1,n+1)-(t-1)+1+(n-1)=\beta-t.
\end{equation}
Moreover, if $(E,V)\in \Sigma_t$ and the torsion
sheaf $\tau$ of (\ref{sigmat}) has support
consisting of $t$ distinct points, then
$\Psi^{-1}(E,V)$ consists of precisely $t$ points.
Hence $\Psi$ is generically finite on
$(p^{-1}(X\times S_{t-1,d-1}))$, so (3) follows from (\ref{dim}).

Finally, for (5), suppose $\frac{\beta}{n+1}$ is an integer and
let $S'$ be any irreducible component of $S_{t_1-1,d-1}$; by
inductive hypothesis, $\dim S'=\beta(n,d-1,n+1)-(t_1-1)$. As in
(\ref{dim}), we have
$$
\dim(\Psi(p^{-1}(X\times S'))=\beta-t_1.
$$
The result follows.
\end{proof}

\section{The Top Critical Value }\label{top}

In the previous sections  we gave a description
of $G_L(n,d,n+1)$. We shall show now that the bound of
Theorem \ref{t1} is best
possible if $\alpha _l >0$ and analyse what happens at this value
of the parameter. Note that the condition $\alpha_l>0$ implies that $n\ge2$.

\begin{theorem}\label{t2a}
Suppose $\alpha_l>0$. Then there exists a coherent
system $(E,V)$ which
is $\alpha_l^+$-stable and $\alpha_l$-semistable,
but not $\alpha_l$-stable.
\end{theorem}

\begin{proof}
We shall construct $(E,V)$ as an extension
\begin{equation}\label{ev}
0\rightarrow(E_1,V_1)\rightarrow(E,V)\rightarrow(E_2,V_2)\rightarrow0,
\end{equation}
where
\begin{itemize}\item[(\ref{ev}a)] $(E_2,V_2)\in G_L(n-1,d_2,n)$ with $d_2=g+n-1-\left[\frac{g}{n}\right]$;
\item[(\ref{ev}b)] $(E_1,V_1)$ is of type $(1,d-d_2,1)$.
\end{itemize}
Note that $d>d_2$ by (\ref{alphal}), so $(E_1,V_1)$ exists.
Moreover $\beta(n-1,d_2,n)\ge0$; so, by Theorem \ref{t1}, $(E_2,V_2)$ also exists and indeed is
$\alpha$-stable for all $\alpha>0$ and in particular for
$\alpha=\alpha_l$. It is easy to check from the definition (\ref{alpha0}) that
\begin{equation}\label{eqn:mu}
\mu_{\alpha_l}(E_1,V_1)=\mu_{\alpha_l}(E_2,V_2),
\end{equation}
so $(E,V)$ is $\alpha_l$-semistable but not $\alpha_l$-stable.
Moreover, since $(E_1,V_1)$ and $(E_2,V_2)$ are both $\alpha_l$-stable but not isomorphic, it follows from (\ref{eqn:mu}) that
\begin{equation}\label{eqn:hom=0}
{\rm Hom}((E_1,V_1),(E_2,V_2))=0={\rm Hom}((E_2,V_2),(E_1,V_1)).
\end{equation}

Now any subsystem of $(E,V)$ which contradicts
$\alpha_l^+$-stability must also contradict
$\alpha_l$-stability. If the extension (\ref{ev})
is non-trivial, the only
subsystem which contradicts $\alpha_l$-stability
is $(E_1,V_1)$ and clearly this does not
contradict $\alpha_l^+$-stability. It remains
only to prove that there exists a non-trivial
extension (\ref{ev}), or equivalently to prove that
$$
{\rm Ext}\,^1((E_2,V_2),(E_1,V_1))\ne0.
$$

Now, by (\ref{eqn:200}) and (\ref{c12}),
$$\dim{\rm Ext}\,^1((E_2,V_2),(E_1,V_1))\ge C_{21}=(k_2-n_2)(d_1-n_1(g-1))+n_1d_2-k_1k_2.
$$
Here we have $(n_1,d_1,k_1)=(1,d-d_2,1)$,
 $(n_2,d_2,k_2)=(n-1,d_2,n)$, so
$$
C_{21}=(d-d_2-g+1)+d_2-n=d-g-n+1.
$$
Since $\alpha_l>0$, it follows from
(\ref{alphal}) that $d-g-n>0$ and so
$C_{21}>0$ as required.
\end{proof}

\begin{corollary}\label{flipl+}
If $\alpha_l>0$, then it is equal to the top
critical value $\alpha_L$. Moreover the flip
locus $G_L^+$ is given precisely by the
non-trivial extensions {\em(\ref{ev})}
 which satisfy {\em(\ref{ev}a)} and {\em(\ref{ev}b)}
and has dimension $\le\beta-1$.

\end{corollary}

\begin{proof} The fact  that $\alpha_L=\alpha_l$ follows at once from Theorems \ref{t1} and \ref{t2a}. If $(E,V)\in G_L^+$, we have a sequence (\ref{eqif}) for which $(E_2,V_2)$ is $\alpha_l^+$-stable and (\ref{eqn:alpha}) holds with $\alpha_i=\alpha_l$. By Lemma \ref{l1}, we must have $k_2\ge n_2+1$ and $d_2\ge g+n_2-\left[\frac{g}{n_2+1}\right]$. By Remark \ref{rmk:200}, it follows that
$$n_2=n-1,\ \ k_2= n,\ \ d_2 =g+n-1-\left[\frac{g}{n}\right].$$
Hence all the conditions of (\ref{ev}) hold.

According to Lemma \ref{alpha+} and Remark \ref{rmk100},
it remains to prove that
$$
C_{12}-h^0(E_1^*\otimes N_2\otimes K)>0.
$$
Putting in values from (\ref{ev}), we have, since $\alpha_l>0$,
$$
C_{12}=(n-1)\left(d-g-n+1+\left[\frac{g}{n}\right]\right)
-n>g-\left[\frac{g}{n}\right]-1
\ge0.
$$
On the other hand, $E_1^*\otimes N_2\otimes K$ is
a line bundle of degree $2g-2-d$. If $d>2g-2$, we
are finished. If $d\le 2g-2$, then, by Clifford's Theorem,
$$
h^0(E_1^*\otimes N_2\otimes K)\le g-\frac{d}2<g-\frac{g+n}{2}.
$$
It is therefore sufficient to prove that
$$
\frac{g+n}2\geq \left[\frac{g}{n}\right]+1.
$$
Since $n\ge2$, this is obvious.
\end{proof}

\begin{remark}\label{re5}\begin{em} The estimate
for the dimension of $G_L^+$ in the proof of Corollary
\ref{flipl+} is sufficient for our purposes, but
is quite crude and can certainly be improved.

\end{em}\end{remark}

We now turn to the determination of the flip locus $G_L^-$.
\begin{proposition}\label{flipl-}
If $\alpha_l>0$, then the flip locus $G_L^-$
consists of the non-trivial extensions
\begin{equation}\label{eqn:201}
0\rightarrow(E_2,V_2)\rightarrow(E,V)
\rightarrow(E_1,V_1)\rightarrow0,
\end{equation}
where $(E_1,V_1)$ and $(E_2,V_2)$ satisfy the
same properties as
in {\em(\ref{ev})}, and has dimension
$\le\beta-1$.
\end{proposition}

\begin{proof}

If $(E,V)\in G_L^-$, then there certainly exists a non-trivial extension (\ref{eqn:201}) with $(E_2,V_2)$ $\alpha_l^-$-stable and 
$$\mu_{\alpha_l}(E_2,V_2)=\mu_{\alpha_l }(E,V),\ \ k_2\ge n_2+1$$ 
(see (\ref{eqn:alpha})). By \cite[Theorem 1(1)]{le}, we must have $\beta(n_2,d_2,n_2+1)\ge0$ and so, 
by Remark \ref{r1}, $d_2\ge g+n_2-\left[\frac{g}{n_2+1}\right]$. By Remark \ref{rmk:200}, it follows that
$$n_2=n-1,\ \ k_2=n,\ \ d_2= g+n-1-\left[\frac{g}{n}\right].$$
Hence all the conditions of (\ref{ev}) hold.
 Now note that $N_1=0$
and $C_{21}>0$ as shown in the proof of Theorem \ref{t2a}.
The proposition follows from Remark \ref{rmk100}.
\end{proof}

\begin{remark}\label{re6}\begin{em} Taking $\alpha=\alpha_l$ in the
proof of Proposition \ref{pe3} gives a slightly
different description of $G_L^-$, namely
$$G_L^-=\{(E,V)\,|\,(E,V)\ {\rm generically\ generated,}\
E'\cong\mathcal{O}\oplus G, H^0(G^*)=0,G\
{\rm saturated\ in}\ E\}.$$
It is easy to see that these two
descriptions are equivalent.
\end{em}\end{remark}

\begin{theorem}\label{theoflip} Suppose $\alpha_l>0$. Then $G_{L-1}$ is
non-empty and irreducible, and is birational to $G_L$.
\end{theorem}

\begin{proof} This follows from Corollary \ref{flipl+}, Proposition \ref{flipl-} and (\ref{eqn:iso}).
\end{proof}

\section{Moduli spaces for any $\alpha $}\label{anyalpha}

As we have seen (see Theorems \ref{t1} and \ref{theoflip}), for
$\beta(n,d,n+1)\ge 0$ and $\alpha > \alpha_{L-1}$, the moduli
space $G(\alpha ;n,d,n+1)$ is non-empty and the non-emptiness is
related with the
 existence of coherent systems $(E,V)$
such that $E$ is generically generated and $H^0(E'^*)=0.$ Our
object in this section is to try to generalise these results to
arbitrary $\alpha>0$. For $d\le g+n$, these results are largely
contained in the unpublished \cite{bu} (see also \cite{bu1}) and in \cite{le}.

We begin by recalling the results of \cite{le} which we require.
\begin{proposition}\label{prop1} \cite[Theorem 1(1)]{le}
Let $X$ be a Petri curve and $\beta<0$. Then $G(\alpha)=\emptyset$
for all $\alpha>0$.
\end{proposition}

Before proceeding further, we define
$$U(n,d,n+1):=\{(E,V)\in G_L: E \mbox{ is stable}\}$$
and
$$U^s(n,d,n+1):=\{(E,V): (E,V) \mbox{ is $\alpha$-stable for
all } \alpha>0\}.$$
Note that $U(n,d,n+1)$ can be defined alternatively as
$$U(n,d,n+1):=\{(E,V): E \mbox{ is stable and } (E,V) \mbox{ is $\alpha$-stable for
all } \alpha>0\}$$ and in particular $U(n,d,n+1)\subset U^s(n,d,n+1)$. In the
converse direction, note that, if $(E,V)\in U^s(n,d,n+1)$, then $E$ is
semistable. However it is not generally true that $U(n,d,n+1)=U^s(n,d,n+1)$ and we
can have $U^s(n,d,n+1)\ne\emptyset$, $U(n,d,n+1)=\emptyset$. Our main object in the
remainder of the paper is to determine when these sets are
non-empty.

\begin{remark}\label{re9}\begin{em} By openness of $\alpha$-stability, $U(n,d,n+1)$ and $U^s(n,d,n+1)$ are open subsets of $G_L$,
thus inheriting natural structures of smooth variety, and with
these same structures they are  also embedded as open subsets of every
$G(\alpha)$.
If either $U(n,d,n+1)$ or $U^s(n,d,n+1)$ is non-empty,
then, by Theorem \ref{t1}, it is irreducible of
dimension $\beta$ (or finite when $\beta=0$) and its generic element $(E,V)$ is generated
with $H^0(E^*)=0$. \end{em}\end{remark}

\begin{proposition}\label{lem1} \cite[Proposition 2.5(4)]{le}
Let $(E,V)$ be a generated coherent system of type $(n,d,n+1)$
such that $E$ is semistable. Then $(E,V)\in U^s(n,d,n+1)$.
\end{proposition}

\begin{proposition}\label{th1} \cite[Proposition 4.1(2)]{le} Let $X$ be a Petri curve and suppose
that $g+n-\left[\frac{g}{n+1}\right]\le d\le g+n$ and that $g$ and
$n$ are not both equal to $2$. Then $U(n,d,n+1)$ is non-empty.
\end{proposition}

\begin{proposition}\label{pt1} \cite[Proposition 4.6]{le} Let $X$ be a Petri curve
and  $\beta\geq 0.$ If $g\geq n^2-1$, then $U(n,d,n+1)\ne\emptyset$.
\end{proposition}

In the remainder of this section, we shall introduce two further
techniques for constructing coherent systems. The first is that of
elementary transformations, which we shall use in two distinct
ways.

Since any stable bundle of degree $\ge n(2g-1)$ is generated by
its sections, Proposition \ref{lem1} implies that $U(n,d,n+1)\ne\emptyset$
for $d\ge n(2g-1)$ (see also \cite[Proposition 2.6]{le}). The next proposition provides a significant
improvement on this.

\begin{proposition}\label{prop2}
Let $X$ be a Petri curve. If
$$d_0=\left\{\begin{array}{ll}\frac{n(g+3)}2&\mbox{if $g$ is
odd}\\\frac{n(g+2)}2&\mbox{if $g$ is even,}\end{array}\right.$$
then $U^s(n,d_0,n+1)\ne\emptyset$. 

If $d\ge d_1$, where
$$d_1=\left\{\begin{array}{ll}\frac{n(g+3)}2+1&\mbox{if $g$ is
odd}\\\frac{n(g+2)}2+1&\mbox{if $g$ is even and }n\le\frac{g!}{(\frac{g}2)!(\frac{g}2+1)!}
\\\frac{n(g+4)}2+1&\mbox{if $g$ is even and }n>\frac{g!}{(\frac{g}2)!(\frac{g}2+1)!},
\end{array}\right.$$
then $U(n,d,n+1)\ne\emptyset$.
\end{proposition}
\begin{proof}
It is easy to check that, with the above definition of $d_0$,
$\beta(1,\frac{d_0}n,2)\ge0$ (in fact, $\frac{d_0}n$ is the smallest
integer for which this is true). Hence, by classical Brill-Noether
theory, there exists a line bundle ${\mathcal L}$ of degree
$\frac{d_0}n$ such that $h^0({\mathcal L})\ge2$ and ${\mathcal L}$ is generated by its sections.
Now let ${\mathcal L}_1,\ldots,{\mathcal L}_n$ be any such line bundles and let $V$ be a subspace of
$H^0({\mathcal L}_1\oplus\ldots\oplus {\mathcal L}_n)$ of dimension $n+1$ such that
$({\mathcal L}_1\oplus\ldots\oplus {\mathcal L}_n,V)$ is generated. Hence
$({\mathcal L}_1\oplus\ldots\oplus {\mathcal L}_n,V)\in U^s(n,d_0,n+1)$ by Proposition \ref{lem1}.

Again by classical Brill-Noether theory, one can find pairwise non-isomorphic line bundles ${\mathcal L}_1,\ldots,{\mathcal L}_n$ of degree $\frac{d_1-1}n$ such that, for all $i$, $h^0({\mathcal L}_i)\ge2$ and ${\mathcal L}_i$ is generated by its sections (in the case $g$ even and $d_1=\frac{n(g+2)}2+1$, the number of distinct line bundles of degree $\frac{d_1-1}n$ with 
$h^0\ge2$ is $\frac{g!}{(\frac{g}2)!(\frac{g}2+1)!}$ \cite[Chapter V, formula (1.2)]{arb}).
Now consider extensions
$$0\rightarrow {\mathcal L}_1\oplus\ldots\oplus {\mathcal L}_n\rightarrow
E\rightarrow\tau\rightarrow0,$$ where $\tau$ is a torsion sheaf of
length $t\ge1$. These extensions are classified by $n$-tuples
$(e_1,\ldots,e_n)$ with $e_i\in{\rm Ext}^1(\tau,{\mathcal L}_i)$. It can be
shown (see \cite[Th\'eor\`eme A.5]{mer}) that, for any $t$, there exists an
extension of this type for which $E$ is stable. Moreover $V$ can
be regarded as a subspace of $H^0(E)$, making $(E,V)$ a coherent
system. If $(E_1,V_1)$ is a proper subsystem of $(E,V)$ with $E_1\ne E$, then $V_1\subset
V\cap H^0(E_1\cap {\mathcal L}_1\oplus\ldots\oplus {\mathcal L}_n)$. It
follows from the $\alpha$-stability of $({\mathcal L}_1\oplus\ldots\oplus
{\mathcal L}_n,V)$ for large $\alpha$ that $\frac{k_1}{n_1}\le\frac{k}{n}$.  Since $E$ is stable, we have
also $\frac{d_1}{n_1}<\frac{d}n$. It follows that $(E,V)\in U(n,d,n+1)$.
\end{proof}

\begin{remark}\begin{em}\label{rmk:te}
For a general curve $X$, the second part of Proposition \ref{prop2} is valid with 
$$d_1=\left\{\begin{array}{ll}\frac{n(g+1)}2+1&\mbox{if $g$ is
odd}\\\frac{n(g+2)}2+1&\mbox{if $g$ is even}\end{array}\right.$$
by \cite{te2}. However, this does not imply the result for an arbitrary Petri curve.
\end{em}\end{remark}

Our second use of elementary transformations is to prove

\begin{proposition}\label{prop4}
Suppose that $U(n,na,n+1)\ne\emptyset$ for some integer $a$. Then
$U(n,d,n+1)\ne\emptyset$ for all $d$ with $d>na$ and
$d\equiv\pm1\bmod n$.
\end{proposition}

\begin{proof}
In view of Remark \ref{R8}, it is sufficient to prove this for
$d=na+1$ and for $d=na+n-1$.

Suppose first that $d=na+1$. Let $(F,V)\in U(n,na,n+1)$ and
define $E$ as an elementary transformation (\ref{eqn2}). Then
$(E,V)\in G_L(n,na+1,n+1)$ by Lemma \ref{lem:200}. The stability of $E$
follows easily from the stability of $F$, so $(E,V)\in U(n,d,n+1)$.

Now suppose $d=na+n-1$. Again let $(F,V)\in G_{L}(n,na,n+1)$ and
let $x\in X$. Let $\tau$ be the torsion sheaf of length $1$
supported at $x$ and define $E$ as an elementary transformation
$$0\rightarrow E\rightarrow F(x)\rightarrow\tau\rightarrow0.$$
Then $F$ can be regarded as a subsheaf of $E$ and $V$ as a
subspace of $H^0(E)$. By Lemma \ref{lem:200}, the coherent system
$(E,V)\in G_L(n,na+n-1,n+1)$. The stability of $E$ follows from the stability of
$F(x)$.
\end{proof}

The second technique is the use of extensions of coherent systems. The idea is to
take a generic element $(E,V)$ of $G_L$ and try to prove that $E$
is stable. If this is not the case, there exists a quotient $E_2$
of $E$ with $\mu(E_2)\le\mu(E)$ and we can choose $E_2$ to be
stable. We have therefore an extension
$$0\rightarrow E_1\rightarrow E\rightarrow E_2\rightarrow 0,$$
and, taking $V_1=V\cap H^0(E_1)$ and $V_2=V/V_1$, we obtain an
extension of coherent systems

\begin{equation}\label{eqn2'}
0\rightarrow(E_1,V_1)\rightarrow(E,V)\rightarrow(E_2,V_2)\rightarrow0.
\end{equation}
We are assuming that $(E,V)$ is a generic element of $G_L$, so
$(E,V)$ is generated and $H^0(E^*)=0$. Using Lemma \ref{l1}, we
see that (\ref{eqn2'}) is subject to the following conditions:
\begin{itemize}
\item $\mu(E_2)\le\mu(E)$;
\item $E_2$ is stable, $(E_2,V_2)$ is generated and $k_2\ge
n_2+1$;
\item
$\mu(E_2)\ge1+\frac1{n_2}\left(g-\left[\frac{g}{n_2+1}\right]\right)$.
\end{itemize}

\begin{proposition}\label{prop5} Suppose that $X$ is a Petri curve, $n\ge3$, $d<g+n+\frac{g}{n-1}$
and  $n_2\le n-2$. Then no extension (\ref{eqn2'}) exists
satisfying the stated conditions.
\end{proposition}
\begin{proof}
Suppose we have such an extension. Then
$$1+\frac1{n_2}\left(g-\left[\frac{g}{n_2+1}\right]\right)\le\mu(E_2)\le\frac{d}n.$$
By Lemma \ref{l2}, the left hand side of this inequality is a
decreasing function of $n_2$; so we have
$$1+\frac1{n-2}\left(g-\left[\frac{g}{n-1}\right]\right)\le\frac{d}n,$$
i.e.
\begin{eqnarray*}
d&\ge&g+n+\frac{2g}{n-2}-\frac{n}{n-2}\left[\frac{g}{n-1}\right]\\
&\ge&g+n+\frac{2g}{n-2}-\frac{ng}{(n-2)(n-1)}\\
&=&g+n+\frac{g}{n-1}.
\end{eqnarray*}
This gives the required contradiction.
\end{proof}

It remains to consider the extensions (\ref{eqn2'}) for which
$n_2=n-1$. We have two cases:
\begin{equation}\label{eqn3}
0\rightarrow(E_1,V_1)\rightarrow(E,V)\rightarrow(E_2,V_2)\rightarrow0,\
\ n_1=k_1=1
\end{equation}
and
\begin{equation}\label{eqn4}
0\rightarrow(E_1,0)\rightarrow(E,V)\rightarrow(E_2,V_2)\rightarrow0,\
\ n_1=1.
\end{equation}
\begin{proposition}\label{prop6}
Suppose that $X$ is a Petri curve, $n\ge2$ and $d>g+n$. Then the extensions (\ref{eqn3})
which satisfy the conditions stated above depend on at most
$\beta-1$ parameters.
\end{proposition}

\begin{proof}
Since $E_2$ is stable and $(E_2,V_2)$ is generated, $(E_2,V_2)\in
G_L(n_2,d_2,n_2+1)$ by Proposition \ref{lem1}. Hence $(E_2,V_2)$
depends on $\beta(n_2,d_2,n_2+1)$ parameters, while $(E_1,V_1)$ depends on
$d_1=\beta(1,d_1,1)$ parameters. By Remark \ref{rmk100},
$$\mathbb{H}^0_{21}=Hom((E_2,V_2),(E_1,V_1))=0.$$
By Lemma \ref{alpha+}, it remains to prove that

\begin{equation}\label{eqn5}
C_{12}>\dim\mathbb{H}^2_{21}.
\end{equation}
Now, by (\ref{c12}),
$$C_{12}=(n-1)d_1-n,$$
while
$$\dim\mathbb{H}^2_{21}=h^0(E_1^*\otimes N_2\otimes K),$$
where $N_2$ is the kernel of the evaluation map
$V_2\otimes{\mathcal O}\rightarrow E_2$. Now $E_1^*\otimes
N_2\otimes K$ is a line bundle of degree $2g-2-d$. If $d\le 2g-2$,
then, by Clifford's Theorem,
$$h^0(E_1^*\otimes N_2\otimes K)\le g-1-\frac{d}2+1=g-\frac{d}2.$$
So (\ref{eqn5}) holds if
$$(n-1)d_1-n>g-\frac{d}2.$$
Since $d_1\ge\frac{d}n$, this will be true if
$$\frac{(n-1)d}n-n>g-\frac{d}2,$$
i.e. if
$$\frac{3n-2}{2n} d>g+n.$$
This is certainly true since $d>g+n$.

If $d>2g-2$, then $h^0(E_1^*\otimes N_2\otimes K)=0$ and we
require to prove only that $C_{12}>0$. In fact
$$C_{12}=(n-1)d_1-n\ge\frac{n-1}n
d-n>\frac{n-1}n(g+n)-n=\frac{n-1}n g-1\ge0.$$
\end{proof}

\begin{remark}\label{re10}\begin{em}
Propositions \ref{prop5} and \ref{prop6} are directed towards
proving that $U(n,d,n+1)\ne\emptyset$. If we wish only to prove that
$U^s(n,d,n+1)\ne\emptyset$, we are not concerned with the stability of $E$
and we need to consider extensions (\ref{eqn3}) under the usual
conditions of \cite[section 6.2]{bomn} for the flip loci $G_i^+$.
We can still assume
that $(E,V)$ is generated with $H^0(E^*)=0$, so  $(E_2,V_2)$ is
also generated with $H^0(E_2^*)=0$, hence $d_2\ge
g+n_2-\left[\frac{d}{n_2+1}\right]$, and now $\mu(E_2)<\mu(E)$. So the result of
Proposition \ref{prop5} holds under the assumption $d\le
g+n+\frac{g}{n-1}$. In Proposition \ref{prop6}, note that
$(E_2,V_2)\in G_L(n_2,d_2,n_2+1)$ by Theorem \ref{t1}(3); so $(E_2,V_2)$ depends
on precisely $\beta(n_2,d_2,n_2+1)$ parameters and the rest of the proof goes
through.

\end{em}\end{remark}

We turn now to the consideration of the extensions (\ref{eqn4}). 

\begin{proposition}\label{prop7}
Let $X$ be a Petri curve and $n\ge3$. Suppose that $d<g+n+\frac{g}{n-1}$. Then there exist no extensions
(\ref{eqn4}) satisfying the conditions of (\ref{eqn2'}) with
\begin{equation}\label{eqn16}
\frac{d}n<\frac{2g}{2n-1}+2.
\end{equation}
\end{proposition}

\begin{proof} Since $(E_2,V_2)$ is generated, we can
write as usual
$$0\rightarrow N_2\rightarrow V_2\otimes{\mathcal O} \rightarrow
E_2\rightarrow0.$$
Note that $H^0(N_2)=0$ and that $(N_2^*,V_2^*)$ is generated. Moreover $N_2^*$ has rank $2$ and, since $h^0(E_2^*)=0$, $h^0(N_2^*)\ge n+1$. Suppose we prove that, for any line subbundle ${\mathcal L}_1$ of $N_2^*$,
\begin{equation}\label{eqnl1}
h^0({\mathcal L}_1)\le1.
\end{equation}
Then, by \cite[Lemma 3.9]{pr},
$$h^0(\det N_2^*)\ge 2n-1.$$
Hence, by classical Brill-Noether theory and the assumption $\mu(E_2)\le\mu(E)$,
$$\frac{(n-1)d}n\ge d_2=\deg N_2^*\ge\frac{(2n-2)g}{2n-1}+2n-2,$$
which contradicts (\ref{eqn16}).

It remains to prove (\ref{eqnl1}). Consider an exact sequence
$$
0\rightarrow {\mathcal L}_1\rightarrow N_2^* \rightarrow {\mathcal L}_2\rightarrow0.
$$
Since $N_2^*$ is generated, so is ${\mathcal L}_2$. But ${\mathcal L}_2$ is certainly
not trivial since $h^0(N_2)=0$, so $h^0({\mathcal L}_2)=s\ge2$
and
$$\deg {\mathcal L}_2\ge\frac{(s-1)g}s+s-1.$$
If $s<n$, then $h^0({\mathcal L}_1)\ge n+1-s\ge2$ and
$$\deg {\mathcal L}_1\ge\frac{(n-s)g}{n-s+1}+n-s.$$
So
\begin{eqnarray*}
d_2=\deg N_2^*&\ge&\frac{(s-1)g}s+s-1+\frac{(n-s)g}{n-s+1}+n-s\\
&=&2g-\frac{(n+1)g}{s(n-s+1)}+n-1.
\end{eqnarray*}
Since $2\le s\le n-1$, this gives
\begin{equation}\label{eqn12}
d_2\ge 2g-\frac{(n+1)g}{2(n-1)}+n-1\ge g+n-1;
\end{equation}
since $\frac{(n-1)d}n\ge d_2$, this contradicts the assumption that 
$d<g+n+\frac{g}{n-1}$. It follows that $s\ge n$, so
$$\deg {\mathcal L}_2\ge\frac{(n-1)g}n+n-1$$
and
\begin{equation}\label{eqn14}
\deg
{\mathcal L}_1=d_2-\deg{\mathcal L}_2<g+n-1-\frac{(n-1)g}n-n+1=\frac{g}n.
\end{equation} The inequality (\ref{eqnl1}) now follows
from classical Brill-Noether theory. This completes the proof.
\end{proof}

\begin{remark}\label{re11}\begin{em}
The non-strict inequality
\begin{equation}\label{eqn15}
d\le g+n+\frac{g}{n-1}
\end{equation}
is sufficient except when $n=3$, when (\ref{eqn12}) fails to
give a contradiction. The
other place where the inequality $d<g+n+\frac{g}{n-1}$ is used
is (\ref{eqn14}). In this case (\ref{eqn15}) gives $\deg {\mathcal L}_1\le\frac{g}n,$ which is sufficient for (\ref{eqnl1}). In particular,
if $n\ge4$, (\ref{eqn15}) and (\ref{eqn16}) are
sufficient for the validity of Proposition \ref{prop7}.
\end{em}\end{remark}

\section{The cases $n=2$, $n=3$ and $n=4$ }\label{n=23}
In this section we shall assume that $g\ge3$.

\begin{theorem}\label{th2}
Let $X$ be a Petri curve of genus $g\ge3$. Then
$U(2,d,3)\ne\emptyset$ if and only if $\beta(2,d,3)\ge0$.
\end{theorem}
\begin{proof} This follows at once from Propositions \ref{prop1}
and \ref{pt1}.
\end{proof}

\begin{theorem}\label{th3}
Let $X$ be a Petri curve of genus $g\ge3$. Then
$U(3,d,4)\ne\emptyset$ if and only if $\beta(3,d,4)\ge0$.
\end{theorem}
\begin{proof}
According to Proposition \ref{pt1}, the result holds for $g\ge8$.
For lower values of $g$, the result holds by Proposition \ref{th1}
in the following cases
\begin{itemize}
\item $g=3, d=6$;
\item $g=4, d=6,7$;
\item $g=5, d=7,8$;
\item $g=6, d=8,9$;
\item $g=7, d=9,10$.
\end{itemize}
For $g\ne5$, Proposition \ref{prop4} and Remark \ref{bint} give the
result for all $d\ge g+3-\left[\frac{g}4\right]$, i.e. for all
$\beta\ge0$.

When $g=5$, Remark \ref{R8} gives the result for $d=10,11$ and
Proposition \ref{prop2} for $d\ge13$, leaving only $d=9,12$ open.
For $g=5,d=9$, the inequalities $d<g+n+\frac{g}{n-1}$, $d>g+n$ and
$\frac{d}n<\frac{2g}{2n-1}+2$ are all satisfied and the result
follows from Propositions \ref{prop5}, \ref{prop6} and
\ref{prop7}. Finally, the case $d=12$ now follows using Remark
\ref{R8}.
\end{proof}

\begin{theorem}\label{th4}
Let $X$ be a Petri curve of genus $g\ge3$. Then $U(4,d,5)\ne\emptyset$ if and only if
$\beta(4,d,5)\ge0$.
\end{theorem}

\begin{proof}
Proposition \ref{pt1} gives $U(4,d,5)\ne\emptyset$ for
$g\ge15$. Now Proposition \ref{th1} covers the following cases
\begin{itemize}
\item $g=3, d=7$;
\item $g=4, d=8$;
\item $g=5, d=8,9$;
\item $g=6, d=9,10$;
\item $g=7, d=10,11$;
\item $g=8, d=11,12$;
\item $g=9, d=12,13$;
\item $g=10, d=12,13,14$;
\item $g=11, d=13,14,15$;
\item $g=12, d=14,15,16$;
\item $g=13, d=15,16,17$;
\item $g=14, d=16,17,18$.
\end{itemize}

\noindent Proposition \ref{prop4} now gives the following additional cases
\begin{itemize}
\item $g=4, d=9,11$;
\item $g=5, d=11$;
\item $g=8, d=13$;
\item $g=9, d=15$;
\item $g=10, d=15$;
\item $g=12, d=17$;
\item $g=14, d=19$.
\end{itemize}

\noindent Remark \ref{bint} now completes the argument for $g=10,12,14$.

For other $g$, we try using extensions of coherent systems. Propositions \ref{prop5},
\ref{prop6} and \ref{prop7}, together with Proposition \ref{prop4}, give the following additional cases
\begin{itemize}
\item $g=5, d=10$;
\item $g=6, d=11$;
\item $g=7, d=12,13$;
\item $g=8, d=14$;
\item $g=9, d=14$;
\item $g=11, d=16$;
\item $g=13, d=18$.
\end{itemize}

\noindent Again using Remark \ref{bint}, this completes the argument for
$g=5,7,8,9,11,13$. Moreover, in view of Proposition \ref{prop2}, the only outstanding cases are  $g=3,d=8,9, 10,12$, $g=4, d=10,14$ and $g=6,d=12,16$.

\begin{proposition}\label{prop31} Suppose that $X$ is a Petri curve of genus $3$ and $d=8,9$ or $12$. Then $U(4,d,5)\ne\emptyset$.
\end{proposition}
\begin{proof}
Suppose first that $d=8$. Since $d=2n$, the result then follows from \cite[Theorem 5.4]{bgmmn}. For $d=9$, we now use Proposition \ref{prop4} and, for $d=12$, we apply Remark \ref{R8}.
\end{proof}

\begin{proposition}\label{prop32} Suppose that $X$ is a Petri curve of genus $6$ and $d=12$ or $16$. Then $U(4,d,5)\ne\emptyset$.
\end{proposition}

\begin{proof} In view of Remark \ref{R8}, it is sufficient to prove that $U(4,12,5)\ne\emptyset$. Note that in this case we have
$$12=d=g+n+\frac{g}{n-1}\mbox{ and }\frac{d}n=3<\frac{2g}{2n-1}+2=\frac{12}7+2.$$
Let $(E,V)$ be a generic element of $G_L(4,12,5)$ and suppose that $E$ is not stable. By Remark \ref{re11} and Proposition \ref{prop6}, the only possible form for a destabilising sequence is
\begin{equation}\label{eqn:new}
0\rightarrow(E_1,V_1)\rightarrow(E,V)\rightarrow(E_2,V_2)\rightarrow0,\ \ E_2\mbox{ stable },n_2\le2.
\end{equation}
Moreover, all the inequalities in the proof of Proposition \ref{prop5} must be equalities, which is the case if and only if
$$n_1=n_2=2\mbox{ and }d_1=d_2=6.$$
Since (\ref{eqn:new}) is the only possible form for a destabilising sequence with $E_2$ stable, it follows that 
$E$ is semistable. If $k_2>3$, then \cite[Lemma 3.9]{pr} applies to give $h^0(\det E_2)\ge5$, which would require $d_2\ge9$ by classical Brill-Noether theory, a contradiction. So $k_2=3$ and $k_1=2$.

Since $(E_2,V_2)$ is generated and $h^0(E_2^*)=0$, we have $(E_2,V_2)\in U(2,6,3)$, which has dimension $\beta(2,6,3)=0$. Since $E$ is semistable and $\mu(E_1)=\mu(E)$, $E_1$ is also semistable. Moreover, $(E_1,V_1)$ must be generically generated, otherwise it would have a subsystem $({\mathcal L},V_1)$ with ${\mathcal L}$ a line bundle, contradicting the $\alpha$-stability of $(E,V)$. It follows that any subsystem $({\mathcal L}_1,W_1)$ of $(E_1,V_1)$ with ${\mathcal L}_1$ of rank 1 has $\deg {\mathcal L}_1\le3$ and $\dim W_1\le1$, so $(E_1,V_1)$ is $\alpha$-semistable for all $\alpha>0$. Now, by \cite[Theorem 5.6]{bomn},
\begin{equation*}%\label{eqn:new1}
\dim G_L(2,6,2)=\beta(2,6,2)=9.
\end{equation*}
On the other hand, if $(E_1,V_1)\not\in G_L(2,6,2)$, we have
\begin{equation}\label{eqn:new2}
0\rightarrow({\mathcal L}_1,W_1)\rightarrow(E_1,V_1)\rightarrow({\mathcal L}_2,W_2)\rightarrow0
\end{equation}
with
\begin{equation*} 
\deg {\mathcal L}_1=\deg {\mathcal L}_2=3\mbox{ and }\dim W_1=\dim W_2=1.
\end{equation*}
Moreover, for the extensions (\ref{eqn:new2}), we have, by (\ref{c12}),
\begin{itemize}
\item $C_{21}=3-1=2$;
\item $\dim{\mathbb H}^0_{21}=\dim\mbox{Hom}(({\mathcal L}_2,W_2),({\mathcal L}_1,W_1))\le1$;
\item $\dim{\mathbb H}^2_{21}=0$ by (\ref{h2}),
\end{itemize}
so
$$\dim\mbox{Ext}^1(({\mathcal L}_2,W_2),({\mathcal L}_1,W_1))\le C_{21}+1=3.$$
Since $({\mathcal L}_1,W_1)$ and $({\mathcal L}_2,W_2)$ each depend on $3$ parameters, the extensions (\ref{eqn:new2}) depend on at most
$$3+3+3-1=8<\beta(2,6,2)$$
parameters.

We now consider the extensions (\ref{eqn:new}) with $(E_1,V_1)$, $(E_2,V_2)$ as above. We have, by (\ref{c12}) and (\ref{h2}),
\begin{itemize}
\item $C_{12}=12-6=6$;
\item $\dim{\mathbb H}^2_{21}=h^0(E_1^*\otimes N_2\otimes K)\le3$ by \cite[Theorem 2.1]{bgn} since $E_1^*\otimes N_2\otimes K$ is semistable of rank $2$ and slope
$$-\frac{d_1}2+\deg N_2+\deg K=-3-6+10=1;$$
\item ${\mathbb H}^0_{21}=0$ by Remark \ref{rmk100}.
\end{itemize}
So, by Lemma \ref{alpha+}, the general $(E,V)\in G_L(4,12,5)$ does not admit an extension (\ref{eqn:new}) and we are done.

\end{proof}

\begin{proposition}\label{prop33}
Suppose that $X$ is a Petri curve of genus $3$ or $4$ and $d=10$ or $14$. Then $U(4,d,5)\ne\emptyset$.
\end{proposition}

\begin{proof}
In view of Remark \ref{R8}, it is sufficient to prove that $U(4,10,5)\ne\emptyset$.
Let $(E,V)$ be a generic element of $G_L(4,10,5)$ and suppose that $E$ is not stable. Then we have a destabilising sequence
\begin{equation}\label{eqn:new4}
0\rightarrow(E_1,V_1)\rightarrow(E,V)\rightarrow(E_2,V_2)\rightarrow0
\end{equation}
satisfying the conditions of (\ref{eqn2'}). We have the following possibilities.
\begin{itemize}
\item $n_2=1$: $3\le\mu(E_2)\le\frac52$, which is a contradiction.
\item $n_2=2$: $\frac12(g+1)\le\mu(E_2)\le\frac52$, so $d_2=4$ or $5$ if $g=3$, $d_2=5$ if $g=4$; moreover $k_2\ge3$ and, by \cite{re}, $h^0(E_2)\le\frac72$, so $k_2=3$.
\item $n_2=3$: $2\le\mu(E_2)\le\frac52$, so $d_2=6$ or $7$; moreover $k_2\ge4$ and, by \cite{re}, $h^0(E_2)\le\frac{d_2+3}2$, giving the possibilities $(d_2,k_2)=(6,4), (7,4), (7,5)$.
\end{itemize}

We consider first the case $n_2=3$. If $k_2=4$, we are in the situation of (\ref{eqn3}) and Proposition \ref{prop6} applies. In the remaining case $d_2=7$, $k_2=5$, we have $h^0(\det E_2)=8-g\le5$. So, by \cite[Lemma 3.9]{pr}, $E_2$ possesses either a line subbundle ${\mathcal L}$ with $h^0({\mathcal L})\ge2$ or a subbundle $F$ of rank $2$ with $h^0(F)\ge3$. In the first case, since $E_2$ is stable, we have $\deg {\mathcal L}\le2$, a contradiction. In the second case $d_F:=\deg F\le4$ and any line subbundle of $F$ has $\deg {\mathcal L}\le2$, hence $h^0({\mathcal L})\le1$. It follows that, for any subspace $W$ of $H^0(F)$ of dimension $3$, $(F,W)\in G_L(2,d_F,3)$. Hence, by Theorem \ref{t1}(1), $\beta(2,d_F,3)\ge0$. Since $d_F\le4$, this holds only when $g=3$, $d_F=4$. It follows that $F$ is semistable and, by \cite{re}, $h^0(F)\le3$ and hence $h^0(F)=3$. Note further that $F$ is not strictly semistable, for otherwise we would have a sequence $0\to {\mathcal L}_1\to F\to {\mathcal L}_2\to0$ with $\deg{\mathcal L}_1=\deg{\mathcal L}_2=2$, so that $h^0(F)\le2$. Hence $F$ is stable and $(F,W)\in U(2,4,3)$. Now let $W_1:=H^0(F)\cap V_2$ and consider the exact sequence
\begin{equation}\label{eqn:new5}
0\to(F,W_1)\to(E_2,V_2)\to({\mathcal L},W_2)\to0,
\end{equation}
where $\dim W_1\le3$. If $\dim W_1<3$, then $\dim W_2\ge3$, contradicting the fact that $\deg{\mathcal L}=3$.
So $\dim W_2=2$, $\dim W_1=3$ and
$$(F,W_1)\in U(2,4,3),\ \ ({\mathcal L},W_2)\in U(1,3,2).$$
For the extensions (\ref{eqn:new5}), we have, by (\ref{c12}) and (\ref{h2}),
\begin{itemize}
\item $C_{21}=4-4+6-6=0$;
\item ${\mathbb H}^0_{21}=0$ by Remark \ref{rmk100}; 
\item $\dim{\mathbb H}^2_{21}=h^0(F^*\otimes {\mathcal L}^*\otimes K)^*=0$ since $F^*\otimes {\mathcal L}^*\otimes K$ is stable of degree $-2$.
\end{itemize}
So, by (\ref{eqn:200}), the extension (\ref{eqn:new5}) splits, which contradicts the stability of $E_2$. We have therefore proved that the only possible destabilising sequences for a general $(E,V)$ of type (\ref{eqn:new4}) with $E_2$ stable are those with $n_2=2$.

Suppose then that $n_2=2$. We have $k_2=h^0(E_2)=3$ and we know that $(E_2,V_2)$ is generated and $h^0(E_2^*)=0$, so $(E_2,V_2)\in U(2,d_2,3)$. Suppose now that $E$ is semistable, so that $d_2=5$. Then also $E_1$ is semistable and in fact stable since $\gcd(n_1,d_1)=1$. It follows that any line subbundle ${\mathcal L}$ of $E_1$ has $\deg {\mathcal L}\le2$ and hence $h^0({\mathcal L})\le1$. So $(E_1,V_1)\in U(2,5,2)$. For the extensions (\ref{eqn:new4}), we have, by (\ref{c12}) and (\ref{h2}),
\begin{itemize}
\item $C_{12}=10-6=4$;
\item $\dim{\mathbb H}^2_{21}=h^0(E_1^*\otimes N_2\otimes K)=0$ since $E_1^*\otimes N_2\otimes K$ is stable with slope $<0$;
\item ${\mathbb H}^0_{21}=0$ by Remark \ref{rmk100}.
\end{itemize}
So, by Lemma \ref{alpha+}, the general $(E,V)$ does not admit an extension of this type.

It remains to consider the possibility that $E$ is not semistable. From the above, this can happen only when $g=3$ and we have an extension (\ref{eqn:new4}) with 
$$n_1=n_2=2,\ d_1=6,\ d_2=4,\ k_1=2,\ k_2=3.$$
We certainly have $(E_2,V_2)\in U(2,4,3)$, but we can no longer guarantee that $E_1$ is semistable. However the maximal degree of a line subbundle of $E_1$ is $4$, for otherwise $E$ would have a quotient bundle of rank $3$ and degree $\le5$; this cannot be stable since $E$ has no stable quotient bundles of rank $3$ contradicting the stability of $E$. It follows that $E$ would have either a quotient line bundle of degree $\le1$ or a stable quotient bundle of rank $2$ of degree $\le3$; both of these are impossible (see the itemized list following (\ref{eqn:new4})). Moreover, we can still argue as in the proof of Proposition \ref{prop32} to show that $(E_1,V_1)$ depends on at most $\beta(2,6,2)$ parameters. Now for the extensions (\ref{eqn:new4}), we have, by (\ref{c12}) and (\ref{h2}),
\begin{itemize}
\item $C_{12}=12-6=6$;
\item $\dim {\mathbb H}^2_{21}=h^0(E_1^*\otimes N_2\otimes K)=0$ since $\deg N_2\otimes K=0$ and the maximal degree of a line subbundle of $E_1^*$ is $-2$;
\item ${\mathbb H}^0_{21}=0$ by Remark \ref{rmk100}.
\end{itemize}
The result now follows from another application of Lemma \ref{alpha+}.

\end{proof}

This completes the proof of Theorem \ref{th4}.
\end{proof}

\begin{remark}\label{rmk:400}\begin{em}
In the course of proving Proposition \ref{prop33}, we have shown that there is no 
coherent system $(E_2,V_2)$ of type $(3,7,5)$ on a Petri curve of genus $3$ or $4$ with $E_2$ stable. A 
slight modification of the proof shows that $G(\alpha;3,7,5)=\emptyset$ for all 
$\alpha>0$ and all $g\ge3$ (we have to prove that $E_2$ is stable for all $(E_2,V_2)\in G(\alpha;3,7,5)$). Since $\beta(3,7,5)=17-6g<0$ for $g\ge3$, this is to be expected, but, so far as we are aware, 
it has previously been proved only for $g\ge6$ (see \cite[Theorem 3.9]{le}, where it is shown that, for $k>n$,  $G(\alpha;n,d,k)=\emptyset$ if $\beta(n,d,n+1)<0$; in this case $\beta(3,7,4)=16-3g<0$ if and only if $g\ge6$.).\end{em}
\end{remark}

\section{Low genus}\label{low}

The cases $g=0$ and $g=1$ have been excluded from the earlier
part of this paper since they present special features and have
been handled elsewhere \cite{ln1,ln2}.

For $g=0$, there are no stable bundles of rank $\ge2$, so $U(n,d,n+1)$ is
always empty if $n\ge2$. Moreover, if $d$ is not divisible by $n$, there exist
no semistable bundles; hence $U^s(n,d,n+1)=\emptyset$. For the remaining case,
when $d$ is
divisible by $n$, $U^s(n,d,n+1)\ne\emptyset$ (see \cite[Proposition 6.4]{ln1}).
One may note that in this case $\beta\ge0$ is equivalent to $d\ge n$.

For $g=1$, the moduli spaces $G(\alpha)$ are well understood (see \cite{ln2}). The
results for $U(n,d,n+1)$ and $U^s(n,d,n+1)$ are summarised in the following theorem.

\begin{theorem}\label{genus1}
Let $X$ be a curve of genus $1$ and $n\ge2$. Then
\begin{itemize}
\item $U^s(n,d,n+1)\ne\emptyset$ if and only if $d\ge n+1$;
\item $U(n,d,n+1)\ne\emptyset$ if and only if $d\ge n+1$ and $\gcd(n,d)=1$.
\end{itemize}
\end{theorem}
\begin{proof}
The first part follows from the main theorem of \cite{ln2} and \cite[Remark 6.3]{ln2}.
For the second part, recall that, on an elliptic curve, stable
bundles exist if and only if $(n,d)=1$, and, in this case, all
semistable bundles are stable.
\end{proof}
The condition $d\ge n+1$ here is precisely equivalent to $\beta\ge0$.

For $g=2$, note first that the case $g=n=2$, $d=4$ is a
genuine exception in Proposition
\ref{th1} (see \cite[Lemma 6.6(1)]{le}). More generally, if $E$ is any bundle of rank $n\ge2$ and
degree $2n$ with $h^0(E)\ge n+1$ on a curve of genus $2$, then $E$
cannot be stable. In fact, by Riemann-Roch, we have $h^1(E)\ge1$,
so there exists a non-zero homomorphism $E\rightarrow K$, which
immediately contradicts stability. There do exist semistable
bundles with $h^0(E)\ge n+1$, which can be constructed as in the proof of Proposition \ref{prop2} or by using
sequences
$$0\rightarrow E^*\rightarrow V\otimes{\mathcal O}\rightarrow
{\mathcal L}\rightarrow0$$ with $\deg {\mathcal L}=2n$ and $V$ a subspace of $H^0({\mathcal L})$ of
dimension $n+1$ which generates ${\mathcal L}$; the coherent system $(E,V^*)$
is then $\alpha$-stable for all $\alpha>0$. We deduce
\begin{theorem}\label{genus2}
Let $X$ be a curve of genus 2 and $n\ge2$. Then
\begin{itemize}
\item $U^s(n,d,n+1)\ne\emptyset$ if and only if $d\ge n+2$ (or equivalently $\beta\ge0$);
\item $U(n,d,n+1)\ne\emptyset$ if and only if $d\ge n+2$, $d\ne 2n$.
\end{itemize}\end{theorem}
\begin{proof}
We have $U(n,d,n+1)\ne\emptyset$ in the following cases:
\begin{itemize}
\item $d\ge3n$ by \cite[Proposition 2.6]{le};
\item $d=n+2,\ldots,2n-1$ by \cite[Theorem 5.5]{bgmmn};
\item $d=2n+2,\ldots,3n-1$ by Remark \ref{R8}.
\end{itemize}
Moreover $U^s(n,2n,n+1)\ne\emptyset$ by Proposition \ref{prop2}.
It remains to prove
\begin{itemize}
\item[(i)] $U(n,2n,n+1)=\emptyset$;
\item[(ii)] $U(n,2n+1,n+1)\ne\emptyset$.
\end{itemize}

For (i), we have already remarked that a vector bundle $E$ of rank $n$ and degree $2n$ with $h^0(E)\ge n+1$
cannot be stable (see also \cite[Th\'eor\`eme 2]{mer2}). 

For (ii),  every stable bundle $E$ of rank $n$ and degree $2n+1$ has $h^0(E)\ge n+1$. If we can prove that there exists such a bundle which is generated, we can choose a subspace $V$ of $H^0(E)$ of dimension $n+1$ such that $(E,V)$ is generated. Then $(E,V)\in U(n,d,n+1)$ by Proposition \ref{lem1}.

To show that $E$ is generated, we need to prove that $h^1(E(-x))=0$ for all $x\in X$. Now $E(-x)$ is stable of degree $n+1$ and $E(-x)^*\otimes K$ is stable of degree $n-1$. We consider the Brill-Noether locus $B(n,n-1,1)$. By \cite{sun} or \cite{bgn}, this locus has dimension $\beta(n,n-1,1)$ and hence codimension
$$1-(n-1)+n(g-1)=2$$
in $M(n,n-1)$. It follows that the generic $E\in M(n,2n+1)$ has
$$h^1(E(-x))=h^0(E(-x)^*\otimes K)=0$$
for all $x\in X$ as required.

This completes the proof of the theorem.
\end{proof}

\begin{theorem}\label{genus3}
Let $X$ be a Petri curve of genus $3$ and $n\ge2$. Then $U(n,d,n+1)\ne\emptyset$ if $\beta\ge0$, except possibly when $n\ge5$, $d=2n+2$.
\end{theorem}
\begin{proof}
For $n=2,3,4$, this has already been proved. For $n\ge5$, we have
$U(n,d,n+1)\ne\emptyset$ in the following cases:
\begin{itemize}
\item $d\ge3n+1$ by Proposition \ref{prop2};
\item $d=n+3,\ldots,2n$ by \cite[Theorem 5.4]{bgmmn};
\item $d=2n+1$ by Proposition \ref{prop4};
\item $d=2n+3,\ldots 3n$ by Remark \ref{R8}.
\end{itemize}
\end{proof}

\begin{remark}\label{rmk:te2}\begin{em}
For general $X$ (but not necessarily for all Petri $X$), the exception can be removed using Teixidor's degeneration methods \cite{te2}.
\end{em}\end{remark}

\begin{remark}\begin{em}\label{rk20}
For $g=4,5$ and $n\ge5$, a similar argument works with the following possible exceptions
\begin{itemize}
\item $g=4$, $d=2n+2,2n+3,3n+2,3n+3$;
\item $g=5$, $n=5$, $d=12,13,17,18$;
\item $g=5$, $n\ge6$, $d=2n+2,2n+3,2n+4,3n+2,3n+3,3n+4$.
\end{itemize}
For general $X$, one can use Teixidor's result to rule out some of the exceptions.
\end{em}\end{remark}

\section{Applications to Brill-Noether theory and dual spans}\label{apli}

We recall from section \ref{prelim} that the {\em Brill-Noether locus} $B(n,d,k)$ and $\widetilde{B}(n,d,k)$
are defined by
$$B(n,d,k)=\{ E\in M(n,d)| h^0(E)\geq k\}$$
and
$$\widetilde{B}(n,d,k)=\{ [E]\in \widetilde{M}(n,d)| h^0(\mbox{gr}(E))\geq k\},$$
It follows that the  formula $(E,V)\mapsto[E]$ defines a morphism
$$\psi : G_0(n,d,k)\rightarrow \widetilde{B}(n,d,k),$$
whose image contains $B(n,d,k)$.

The following theorem, which is essentially a restatement of \cite[Theorem 11.4 and Corollary 11.5]{bomn} for the case $k=n+1$, is true for any smooth curve; we state it in a very general and formal way to make it applicable in a wide variety of situations.

\begin{theorem}\label{bn} Suppose that $B(n,d,n+1)\ne\emptyset$. Then
\begin{itemize}
\item[(1)] $\psi$ is one-to-one over $B(n,d,n+1)-B(n,d,n+2)$;
\item[(2)] if  $G_0(n,d,n+1)$ is irreducible, then $B(n,d,n+1)$ is irreducible;
\item [(3)] if $\beta(n,d,n+1)\le n^2(g-1)$ and $G_0(n,d,n+1)$ is smooth and irreducible, then
$${\rm Sing}B(n,d,n+1)=B(n,d,n+2)$$
and $G_0(n,d,n+1)$ is a desingularisation of the closure $\overline{B}(n,d,n+1)$ of $B(n,d,n+1)$ in $\widetilde{M}(n,d)$.
\end{itemize}
\end{theorem}

\begin{proof} (1) is obvious.

(2) follows from (1) and the fact that $B(n,d,n+1)$ is a Zariski-open subset of $\psi(G_0(n,d,n+1)$. [Note that the hypothesis $\beta(n,d,n+1)\le n^2(g-1)$ of \cite[Conditions 11.3]{bomn} is not needed here.]

(3) follows from \cite[Corollary 11.5]{bomn}.
\end{proof}

Of course, if $U(n,d,n+1)\ne\emptyset$, then $B(n,d,n+1)\ne\emptyset$. Thus we have many instances in this paper for which $B(n,d,n+1)\ne\emptyset$. We shall not list all of them as we shall be stating a more specific result later. For the time being, we note the following two corollaries. The first is a slightly extended version of \cite[Corollary 4.5]{le}, the second is new.

\begin{corollary}\label{cor:bn1} Suppose that $X$ is a Petri curve, $g+n-\left[\frac{g}{n+1}\right]\le d\le g+n$ and $(g,n)\ne(2,2)$. Then
\begin{itemize}
\item[(1)] $B(n,d,n+1)$ is irreducible of dimension $\beta(n,d,n+1)$ and smooth outside $B(n,d,n+2)$;
\item[(2)] $G_L(n,d,n+1)$ is a desingularisation of $\overline{B}(n,d,n+1)$;
\item[(3)] if either $d<g+n$ or $d=g+n$ and $n\not\,\mid g$, $B(n,d,n+1)$ is projective and $G_L(n,d,n+1)$ is a desingularisation of $B(n,d,n+1)$.
\end{itemize}\end{corollary}

\begin{proof} The condition on $d$ implies that $\alpha_l\le0$. Hence, by Theorem \ref{t1}, $G_0(n,d,n+1)=G_L(n,d,n+1)$ and is smooth and irreducible of dimension $\beta(n,d,n+1)$. Moreover $U(n,d,n+1)\ne\emptyset$ by Proposition \ref{th1}. (1) and (2) now follow from Theorem \ref{bn}. For (3), we note that, under the stated conditions on $d$, $E$ is stable for every $(E,V)\in G_L(n,d,n+1)$ \cite[Proposition 3.5]{le}; hence $\psi(G_L(n,d,n+1))=B(n,d,n+1)$.
\end{proof}

\begin{remark}\label{rmk:bn1}\begin{em}
When $g=n=2$ and $d=4$, $B(2,4,3)=\emptyset$ by \cite[Lemma 6.6]{le}, but $G_L(2,4,3)\ne\emptyset$. In this case, the image of $\psi$ is contained in $\widetilde{M}(2,4)\setminus M(2,4)$.
\end{em}\end{remark}

\begin{corollary}\label{cor:bn2} Suppose that $X$ is a Petri curve and that all the flip loci for coherent systems of type $(n,d,n+1)$ have dimension $\le\beta(n,d,n+1)-1$. If $B(n,d,n+1)\ne\emptyset$, then
\begin{itemize}
\item $B(n,d,n+1)$ is irreducible;
\item $B(n,d,n+1)$ is smooth of dimension $\beta(n,d,n+1)$ at $E$ whenever $E$ is generically generated and $h^0(E)=n+1$.
\end{itemize}\end{corollary}

\begin{proof} The hypotheses imply that $G_0(n,d,n+1)$ is birational to $G_L(n,d,n+1)$ and is therefore irreducible. Irreducibility of $B(n,d,n+1)$ follows from Theorem \ref{bn}(2). If $E$ is stable, $h^0(E)=n+1$ and $E$ is generically generated, then $(E,H^0(E))\in U(n,d,n+1)$, which is smooth of dimension $\beta(n,d,n+1)$ by Theorem \ref{t1}(4). The result follows from \cite[Theorem 11.4(iv)]{bomn}.
\end{proof}

We know that this corollary has genuine content since the flip loci at $\alpha_l=\alpha_L$ have dimension $\le\beta(n,d,n+1)-1$ (Corollary \ref{flipl+} and Proposition \ref{flipl-}).

We now turn to our second application. Suppose that ${\mathcal L}$ is a generated line bundle of degree $d>0$ and let $V$ be a linear subspace of $H^0({\mathcal L})$ of dimension $n+1$ which generates ${\mathcal L}$ (in other words, $({\mathcal L},V)$ is a generated coherent system of type $(1,d,n+1)$). We have an evaluation sequence
\begin{equation}\label{eqn100}
0\longrightarrow M_{V,{\mathcal L}}\longrightarrow V\otimes{\mathcal O}\longrightarrow {\mathcal L}\longrightarrow0.
\end{equation}
This is also known as the {\em dual span} construction (see \cite{bu}) and has been used in the context of coherent systems in \cite{bomn,le} and also in the proof of Proposition \ref{pe2}. The following is a special case of \cite[Conjecture 2]{bu}. 

\begin{conjecture}\label{conj}
Let $X$ be a Petri curve of genus $g\ge3$. Suppose that 
$\beta:=\beta(1,d,n+1)\ge0$ and that $\mathcal L$ is a general element of $B(1,d,n+1)$ (when $\beta=0$, $\mathcal L$ 
can be any element of the finite set $B(1,d,n+1)$) and let $V$ be a general subspace of $H^0(\mathcal L)$ of dimension $n+1$. Then $M_{V,{\mathcal L}}$ is stable.
\end{conjecture}

This conjecture is related to our results by the following  simple proposition (compare \cite[Theorem 5.11]{bomn}).

\begin{proposition}\label{thm100}
Suppose that $X$ is a Petri curve. The following are equivalent:
\begin{enumerate}
\item there exists a generated coherent system $({\mathcal L},V)$ of type $(1,d,n+1)$ with $M_{V,{\mathcal L}}$ stable;

\item $U(n,d,n+1)\ne\emptyset$.
\end{enumerate}
\end{proposition}

\begin{proof} For (1)$\Rightarrow$(2), we note that $(M_{V,{\mathcal L}}^*,V^*)$ is a generated coherent system of type $(n,d,n+1)$ with $M_{V,{\mathcal L}}^*$ stable, so $(M_{V,{\mathcal L}}^*,V^*)\in U(n,d,n+1)$ by Proposition \ref{lem1}. Conversely, suppose $U(n,d,n+1)\ne\emptyset$. If $\beta(n,d,n+1)>0$, the generic element of $U(n,d,n+1)$ is a generated coherent system $(E,W)$ with $h^0(E^*)=0$ and $E$ stable. If $\beta(n,d,n+1)=0$, then all elements of $U(n,d,n+1)$ have this property. The dual of the evaluation sequence of $(E,W)$ can be written as
$$0\longrightarrow E^*\longrightarrow W^*\otimes{\mathcal O}\longrightarrow {\mathcal L}\longrightarrow0,$$  
where ${\mathcal L}$ is a line bundle of degree $d$. It follows that $M_{W^*,{\mathcal L}}\cong E^*$ and is therefore stable, proving (1).
\end{proof}

\begin{remark}\label{rmk:500}\begin{em}  By Theorem \ref{genus2} and Proposition \ref{thm100}, the conjecture fails for $g=2$, $d=2n$, but is otherwise true for $g=2$. In fact, although Butler \cite[\S1]{bu} discusses the question of whether $M_{V,{\mathcal L}}$ is stable, \cite[Conjecture 2]{bu} actually has the weaker conclusion that $(M^*_{V,{\mathcal L}},V^*)\in G_0(n,d,n+1)$. In this form the conjecture is true for $g=2$ (see Theorem \ref{genus2}).
\end{em}\end{remark}

Using Proposition \ref{thm100}, we can now begin to form a list of cases for which Conjecture \ref{conj} holds. In the list we have noted where each case was proved.
\begin{itemize}
\item $g+n-\left[\frac{g}{n+1}\right]\le d\le g+n$ (\cite{bu}, \cite[Proposition 4.1]{le});
\item $g\ge n^2-1$ (\cite{bu}, \cite[Proposition 4.6]{le});
\item $d\ge d_1$ (Proposition \ref{prop2}, \cite{te2});
\item $d\le2n$ (\cite{mer1,mer2,bgmmn});
\item $n=3,4$ (Theorems \ref{th3}, \ref{th4})
\end{itemize}
The first and fourth items in this list can be expanded further by the use of Remark \ref{bint} and Proposition \ref{prop4}. According to the analysis in section \ref{n=23}, the following cases for $n=3$ and $n=4$ depend on the use of extensions of coherent systems (possibly in conjunction with other methods):
\begin{itemize}
\item $n=3, g=5, d=9,12$;
\item $n=4, g=3, d=10$;
\item $n=4, g=4, d=10,14$;
\item $n=4, g=5, d=10,14$;
\item $n=4, g=6, d=11,12,15,16$;
\item $n=4, g=7, d=12,13,16,17,20$;
\item $n=4, g=8, d=14,18$;
\item $n=4, g=9, d=14,18,22$;
\item $n=4, g=11, d=16,20,24,28$;
\item $n=4, g=13, d=18,22,26,30$.
\end{itemize}
All of these cases, and those depending on Propositions \ref{prop2} and \ref{prop4}, are (so far as we are aware) new.

Of the methods we have used, the only ones capable of further development appear to be elementary transformations (using direct sums of higher rank vector bundles) and extensions of coherent systems (using more refined calculations). 
The methods of \cite{te2} could also yield improved results for general $X$.

\end{document}